\documentclass{article}
\usepackage{amssymb}
\usepackage{amsmath}
\usepackage{amsthm}

\usepackage{multirow}
\usepackage{tabularx}

\usepackage{diagbox}
\usepackage{graphicx}
\usepackage{subcaption}

\usepackage[dvipsnames]{xcolor}
\usepackage{booktabs}
\usepackage{thmtools, thm-restate}

\graphicspath{{figures/}}

\newtheorem{theorem}{Theorem}
\newtheorem{lemma}{Lemma}

\newcolumntype{L}[1]{>{\raggedright\let\newline\\\arraybackslash\hspace{0pt}}m{#1}}
\newcolumntype{C}[1]{>{\centering\let\newline\\\arraybackslash\hspace{0pt}}m{#1}}
\newcolumntype{R}[1]{>{\raggedleft\let\newline\\\arraybackslash\hspace{0pt}}m{#1}}

\begin{document}


\title{Generalized Multiscale Multicontinuum Model for Fractured Vuggy Carbonate Reservoirs}

%
%


\author{Min Wang\thanks{Department of Mathematics, Texas A\&M University, College Station, TX 77843, USA (\texttt{wangmin@math.tamu.edu}).}
   \and Siu Wun Cheung\thanks{Department of Mathematics, Texas A\&M University, College Station, TX 77843, USA (\texttt{tonycsw2905@math.tamu.edu}).}
\and Eric T. Chung\thanks{Department of Mathematics, The Chinese University of Hong Kong, Shatin, New Territories, Hong Kong SAR, China (\texttt{tschung@math.cuhk.edu.hk}).}
\and Maria Vasilyeva\thanks{Institute for Scientific Computation, Texas A\&M University, College Station, TX 77843 \& Department of
Computational Technologies, North-Eastern Federal University, Yakutsk, Republic of Sakha (Yakutia), Russia,
677980. (\texttt{vasilyevadotmdotv@gmail.com}).}
 \and Yuhe Wang\thanks{Department of Petroleum Engineering, Texas A\&M University at Qatar, Doha, Qatar (\texttt{yuhe.wang@qatar.tamu.edu}).}}


\maketitle

\begin{abstract}
 Simulating flow in a highly heterogeneous reservoir with multiscale characteristics could be considerably demanding. To tackle this problem, we propose a numerical scheme coupling the Generalized Multiscale Finite Element Method (GMsFEM) with a triple-continuum model aimed at a faster simulator framework that can explicitly represent the interactions among different continua. To further enrich the descriptive ability of our proposed model, we combine the Discrete Fracture Model (DFM) to model the local effects of discrete fractures. In the proposed model, GMsFEM, as an advanced model reduction technique, enables capturing the multiscale flow dynamics. This is accomplished by systematically generating an approximation space through solving a series of local snapshot and spectral problems. The resulting eigen-functions can pass the local features to the global level when acting as basis functions in coarse problems. Our goal in this paper is to further improve the accuracy of flow simulation in complicated reservoirs especially for the case when multiple discrete fractures located in single coarse neighborhood and multiscale finite element methods fail. Together with a detailed description of the model, several numerical experiments are conducted to confirm the success of our proposed method. A rigid proof is also given in the aspect of numerical analysis.
\end{abstract}

\section{Introduction}

Simulating fluid flow in a fractured vuggy carbonate reservoir has always been of great interest and challenge to both academia and the petroleum industry. A large share of 
worldwide hydrocarbon resources are stored in such reservoirs. For this reason, it is a long-lasting quest of many scientists and engineers to accurately model the underground flow dynamics for a better understanding of fractured vuggy carbonates. Nevertheless, the simulation of flow within such heterogeneous formations is notorious for natural coexistence of different continua, and distinct scales of underlying porous media. Not only the sizes of vugs are in different orders of magnitude, the fracture configurations are also highly diverse. On top of all, the connectivity of vugs and fractures are complicated, further exacerbating the fluid flow modeling in such media. Mass 
would transfer among different scales and continua 
in different forms, which makes flow modeling a task full of challenges.

Among many tools that have been developed to address the characteristics of such problems, an easy and convenient way to explore the dynamics is to use a fine scale simulation. A fine partition is required to divide the domain into local pieces and then a global description of the flow is obtained by puttting the local solutions together. A common fine scale simulations can be conducted under frameworks such as the Finite Element Method (FEM) \cite{baca1984modelling}, Finite Volume 
Method (FVM) or Finite Difference Method (FDM). 

Yet, due to high complexity of the media, it is impossible by 
nature for a fine simulation to simultaneously resolve all-scale media effects. One would either lose important small-scale information or run into a system that is extremely expensive to solve. Thus, a model-reduction scheme is necessary to design a practical numerical method for flow simulation in complicated porous media such as fractured vuggy carbonates. 

Homogenization is commonly used in many model reduction methods \cite{matache2002two,arbogast1990derivation}. With homogenization, the domain of interest is partitioned into many coarse blocks, and the effective properties are calculated for each coarse block aiming at homogenizing local heterogeneous media using information in finer scale within the block. These pre-computed properties can thus catch and average fine scale characteristics and further calibrate the coarse solution accordingly \cite{durlofsky1991numerical}. 
This up-scaling scheme has been proved to be quite effective for simulations on media with scale separation or periodicity \cite{hassani1998review}, but it fails to model the interaction between matrix, fractures and vugs. 

That is exactly what motivated the multi-continuum model. To represent the flow between different continua, each continua in the domain of interest is now considered as a system that expands the whole domain. For example, a fractured and vuggy heterogeneous carbonate reservoir can now be described as three parallel continuum: fractures, matrix and vugs. Different continua coexists at every spot of the domain while they macroscopically interact with each other. The interaction between different continua is coupled based on the mass conservation law. For each continua, both intra and inter flow transfers are modeled.
The first multi-continuum model, the dual porosity model (DPM), is proposed by Barenblatt for flow through naturally fissured rock \cite{barenblatt1960basic}. In his work, two continuum were proposed to represent low and high porosity continua, respectively \cite{warren1963behavior,thomas1983fractured}. And later, a third continua was introduced by researchers to generalize the application of such an idea \cite{pruess1982practical,zhang2018multiscale,yan2016beyond,yan2016general}. Thanks to the simplicity and flexibility of this method, it has been widely adopted in many fields \cite{li2018multiscale}. However, the limitation of the multi-continuum model is also prominent, as it assumes all continuum are connected globally. This assumption will only be valid when each continua has merely global effects. 
For well-developed fractures, one continua is sufficient to represent 
its effects, yet such continua fails to model, for example, the 
independent long fractures as theses fractures may have local 
responses that can not be globalized.

Such "dicontinuum,'' like discrete fractures, are pretty common in reservoir modeling . This inspires us to also consider a discrete model. 
Ideally, the discrete models should be able to describe medias such as 
fractures that mainly contribute to local flow transfers. One common choice is the Embedded Fracture Model (EDFM). Discrete fractures and rock volume cells are considered as two separate systems, and the flow transfer between them is modeled element by element \cite{lee1999efficient,li2006efficient,chai2018efficient}. EDFM is computationally effective in the sense its discretization of rock (matrix) is independent of the spatial distribution of fractures. Yet, its effectiveness is sensitive to conductivity contrast. Another classic ``dicontinuum'' approach is the Discrete Fracture Model (DFM). DFM, on the other hand, can also be used to describe cross-flow between the fractures and matrix which are spatially adjacent to each other \cite{noorishad1982upstream,karimi2001numerical,akkutlu2017multiscale,akkutlu2018multiscale}. Unlike developed fractures, each discrete fracture is described using a $n-1$ dimension element instead of a $n$-dimensional continua. Researches have shown that using discrete networks can accurately resolve the local characteristics of the media with discrete fractures regardless of the conductivity contrast. Discontinuous models for vugs were also developed which take both free and porous flow into consideration \cite{yao2010discrete}. 
 
In order to overcome the constraints of the homogenization scheme, and enrich the heterogeneous information reserved from the local fine-sacle region, the Multiscale Finite Element Method (MsFEM) was developed \cite{efendiev2011multiscale}. Like homogenization, the domain of interest is first partitioned with a coarse mesh. Each local block is then further partitioned with a finer mesh. With this two-level mesh, 
MsFEM can build a series of basis functions in each local region. These basis functions are obtained by solving a series of local problems which incorporate the heterogeneous characteristics of the local coarse region. The basis functions for all coarse regions will work together as a new global approximation space to replace the standard one used in FEM. With such setting, the size of the resulting numerical system is reduced significantly. MsFEM can be easily coupled with DFM when modeling flow in a fractured reservoir. However, more in-depth investigation indicates that when there are more than one independent fracture network in a coarse local region, MsFEM is not able to restore the local 
dynamics accurately \cite{zhang2018multiscale}.

The Generalized Multiscale Finite Element Method (GMsFEM) was later developed to tackle such a problem. It has been proven that GMsFEM can strengthen the ability of MsFEM on solving multi-scale problems. By conducting a spectral decomposition over the local snapshot space, GMsFEM can identify basis functions corresponding to dominant modes of local heterogeneous regions \cite{efendiev2013generalized,efendiev2011multiscale}. This makes automatic enrichment of the multiscale space possible \cite{efendiev2015hierarchical,wang2017generalized}. In the paper \cite{chung2017coupling}, a one-one correspondence between GMsFEM basis functions and high-conductivity networks was presented. This property of basis makes GMsFEM a necessity when dealing with practical examples like carbonate reservoir simulations, as many fracture channels may coexist in a single local region.
By nature of MsFEM and GMsFEM, the production of basis functions 
are independently conducted within each coarse neighborhood. And we can employ parallel computing to speed up the multiscale basis generation process.

With a GMsFEM framework, a multicontinuum model, and the discrete fracture network, we can inherit the merits of the three, and couple them together to improve the capability as well as accuracy of our simulation. In this paper, we use this fully coupled system to describe the flow in fractured and vuggy heterogeneous reservoirs. Conforming unstructured mesh is used to surrender the random discrete fracture networks. Fractures are treated hierarchically. Highly developed fractures with only global effects are modeled as a fracture continuum, while fractures that have local effects are embedded as discrete fracture networks. For independent vugs, a continuum is used to represent their effects with specific configurations such that no intra-flow is considered. 
GMsFEM allows us to consistently develop an approximation space that contains prominent sub-grid scale heterogeneous background information based on the multi-continuum and DFM.

The paper is organized as follows: in Section \ref{section_prelim}, the problem under discussion is clarified, followed by Section \ref{sec_multicontinuum} which briefly reviews the multi-continuum 
model. In Section \ref{sec_GMsFEM}, a step-by-step illustration on GMsFEM together with a priori error estimate is provided.
The details of time discretization of our problem is 
also discussed in this section. In Section \ref{sec_numerical}, we 
present multiple numerical results to varify the effectiveness of porposed methods. Lastly, this paper is concluded by Section \ref{sec_conclusion}.

\section{Preliminaries}\label{section_prelim}
In this paper, we consider a 2-dimensional flow problem in a multiscale porous media. We assume that the dynamic of flow is governed by the Darcy's Law. For simplicity, we ignore the gravity and the capillary pressure effects.
\subsection{Equation for slightly compressive flow in porous media}
In specific, we consider the following equation for slightly compressive flow in heterogeneous porous media, 
\begin{equation}\label{eq:flow}
\frac{\phi c}{B^{\circ}}\frac{\partial u}{\partial t} -\frac{1}{\mu}\nabla\cdot(\frac{\kappa}{B}\nabla u) = f \qquad \text{in } \Omega.
\end{equation}
Here, $\Omega$ is the computational domain. $c$ is compressibility and $\mu$ is viscosity of the liquid. $B^{\circ}$ is the formation volume factor (FVF) at reference pressure $u^0$ and $B$ is a FVF at reservoir condition. They are used to quantify compressibility of the target liquid. 
$\phi$ represents porosity of the fracture vuggy media, while 
$\kappa$ is a permeability function that bears multiscale features in the media (See Figure \ref{fig:ms_perm} for an illustration).
The solution to be sought is pressure $u$, given a production rate
$f$.
\begin{figure}[h]
	\centering
	\includegraphics[width= 0.5\textwidth]{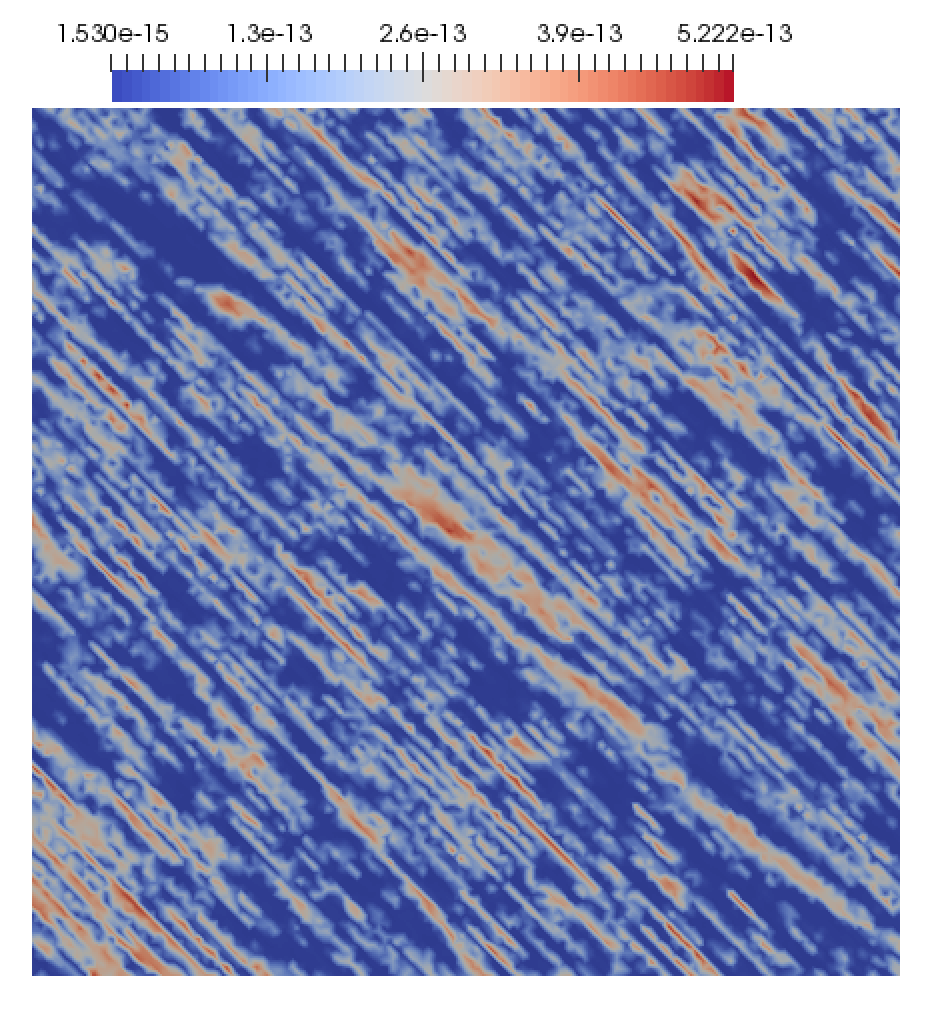}
	\caption{Permeability field $\kappa(x)$ with multiscale features}
	\label{fig:ms_perm}
\end{figure}

Limiting our interests to slightly compressible liquid, we can further employ the simplified correlation between the formation volume factor $B$ and the pressure $u$ 
\begin{equation}\label{eq:FVF}
B = \frac{B^{\circ}}{1+c(u-u^{0})}
\end{equation}
to rewrite \eqref{eq:flow} and get
\begin{equation}\label{eq:flow with B}
\frac{\phi c}{B^{\circ}}\frac{\partial u}{\partial t} -\frac{1}{\mu}\nabla\cdot(\kappa\frac{1+c(u-u^{0})}{B^{\circ}}\nabla u) = f\qquad\text{in }\Omega.
\end{equation}
In the following sections, we will derive our method based on \eqref{eq:flow with B} along with Dirichlet or Neumann boundary conditions on $\partial\Omega$
$$	u = h \quad\text{or}\quad
	-\frac{\kappa}{\mu}\cdot \frac{\partial u}{\partial n} = f. $$

Throughout this paper, we assume $c$, $\phi$ $\mu$ and $B^{\circ}$ are constants. \eqref{eq:FVF} can then be reformulated as 
\begin{equation}\label{eq:PDE}
	b \frac{\partial u}{\partial t}-\nabla\cdot(\kappa(x)\alpha(u)\nabla u) = q	
\end{equation}
where $$\displaystyle b = \frac{\phi c}{B^{\circ}}$$ is a constant, while $$\displaystyle \alpha(u) = \frac{1}{\mu}\cdot(1+c(u-u^{0}))$$ is a field map in $u$.
\subsection{Fine-scale spatial discretization}
For flow in a fractured and vuggy media, the multiscale flow problem described in \eqref{eq:PDE} becomes more complicated as the fractures and vugs have very different hydraulic properties from its background matrix. They can bring in extra transfer and storage mechanics to the flow.
The fractures amplify the complexity of modeling as they can have a wide range of scales and topology. In order to delicately model the their effects on flow, we apply a hierarchical approach. Fractures that have only local effects can be resolved by fine mesh. Thus, the computational domain can be partitioned into
\begin{equation}\label{eq:domain}
	\Omega = \Omega_M\bigoplus_sd_s\Omega_{F,s}\text{, }
\end{equation}
where $M$ and $F$ correspond to matrix and fracture regions respectively. $\Omega_{F,s}$ is a 1-dimensional region that represents one resolved fracture with an aperture $d_i$.
Those fractures that have global effects and are not resolved by mesh can later be handled by representing them as one continua. So are the effects of vugs, which will be discussed in details in next section.
For a fine-scale approximation of $p$, we discretize the PDE on a fine grid, and apply Finite Element Method as well as DFM. Specifically, all integrations in the weak form of \eqref{eq:PDE}, will now be taken separately in both 
$\Omega_M$ and $\Omega_{F,s}$ with distinct hydraulic parameters. To compromise 
arbitrary fractures $\Omega_{F,s}$, one need to adopt an unstructured 
fine-scale mesh. The resulting semi-discrete numerical system is 
\begin{equation}\label{eq:fine_bilinear}
\begin{split}
	&\int_{\Omega_M} b_M\frac{\partial u_h}{\partial t}v_h\ dx + \sum_{s}\int_{\Omega_{F,s}} b_{F,s}\frac{\partial u_h}{\partial t}v_h\ dx \\
	 &+ \int_{\Omega_M}\kappa_M(x) \alpha(p_h)\nabla u_h\nabla v_h\ dx
	 + \sum_s\int_{\Omega_{F,s}}\kappa_{F,s}(x) \alpha(u_h)\nabla_F u_h\nabla_F v_h\ dx\\
	 &=\int_{\Omega}fv_h\ dx. \qquad
\end{split} 
\end{equation}
Here, $v_h$ is a standard FEM basis function. $\nabla_F$ means taking directional derivative along the degenerated fracture $\Omega_{F,i}$.
Note that the aperture effects are considered in $\kappa_{F,s}(x)$.
$$b_M =\frac{\phi_M\ c}{B^{\circ}} \quad b_{F,s} =\frac{\phi_{F,s}\ c}{B^{\circ}} $$
are again constants.
\section{Multi-continuum Model}\label{sec_multicontinuum}
To explicitly represent the global effects of unresolved fractures, vugs and matrix, we introduce the multi-continuum methods. We consider the media as a coupled system of three parallel continua: matrix, unresolved fractures(usually natural fractures), and vugs. They coexist everywhere in our computational domain, while they interact with each other via mass transfer (see Figure \ref{fig:triple-continuum} for an illustration). 
\begin{figure}
	\centering
\includegraphics[width = 0.3\textwidth]{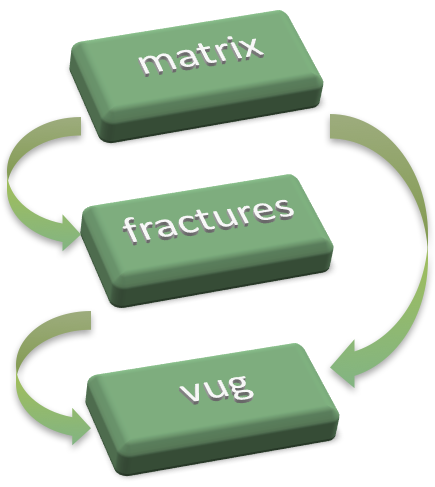}
\caption{Illustration of triple-continuum model}
\label{fig:triple-continuum}
\end{figure}
Without of loss of generality, we assume that all continuum interact with each other. If we denote the flow pressure for continua $i$ as $u^i$, and write the interaction between continua $i$ and $j$ as $Q^{i,j}$, we can then establish a system of PDE following \eqref{eq:PDE} to describe the flow mechanism in each continua.
\begin{equation}\label{eq:multi-con}
		b^i \frac{\partial u^i}{\partial t}-\nabla\cdot(\kappa^i(x)\alpha(u^i)\nabla u^i) = f^i-\sum_{j\neq i}Q^{i,j}	
\end{equation}
and
$$b^i = \frac{\phi^i c}{B^{\circ}}.$$
Here $i$ can be $m$, $f$,or $v$ which stands for matrix, unresolved fractures and vugs respectively.
We further assume that there is no intra-flow inside the vugs and all vugs only 
act as a storage in this system. That is to say, we only consider the case 
when all vugs are independent from each other. Mass transfer due to 
inflow of liquid along vugs can be disregarded in any element of the domain. 
Therefore, mass balance equation for vugs can be written as 
\begin{equation}\label{eq:vug_eq}
	b^v\frac{\partial u^v}{\partial t} = f^v + Q^{f,v}+ Q^{m,v}.
\end{equation}
The term $Q^{i,j}$ represents the mass transfer from continua $i$ to continua $j$. This transfer can be modeled using \cite{warren1963behavior,thomas1983fractured}
$$Q^{i,j} = \sigma\frac{\kappa^{i,j}}{\mu}(u^i-u^j) = q^{i,j}(u^i-u^j),$$
where $q^{i,j} = q^{j,i}$. 
Here $\sigma$ is a shape factor, and $\kappa^{i,j}$ is taken as harmonic mean 
of the permeability $\kappa^i$ and $\kappa^j$.

With \eqref{eq:multi-con} and \eqref{eq:vug_eq}, we can derive the weak formulation of our proposed triple-continuum system. For matrix and unresolved fracture, we have
\begin{equation}\label{eq:mf_variational}
\int_{\Omega} b^{i}\frac{\partial u^i}{\partial t}v^i\ dx+\int_\Omega 
\kappa^i(x)\alpha(u^i)\nabla u^i\nabla v^i\ dx +\sum_{j\neq 
i}\int_{\Omega}Q^{i,j}v^i\ dx= \int_\Omega f^iv^idx, \quad i = m, f.
\end{equation}
For vugs, we have
\begin{equation}\label{eq:v_variational}
	\int_{\Omega} b^{v}\frac{\partial u^v}{\partial t}v^v\ 
dx-\int_{\Omega}Q^{m,v}v^v\ dx - \int_{\Omega}Q^{f,v}v^v\ dx= \int_\Omega 
f^vv^vdx. 
\end{equation}
Here, $v^i$ is any testing function in the same space as $u^i$. We mention that 
equation of $u^m$, $u^f$ and $u^v$ are coupled through term $Q^{i,j}$, thus this 
coupled system should be solved on a Cartesian product space $(u^m,u^f,u^v)\in 
V^m\times V^f\times V^v$. 
In our proposed approach, we take $V^i = H_0^1(\Omega)$ for all continuum $i$.

To express effects of both unresolved and resolved fractures on flow dynamics, we manage to incorporate DFM when solving this multicontinuum equation 
system \eqref{eq:mf_variational}--\eqref{eq:v_variational}. Like what we have in \eqref{eq:fine_bilinear}, we assume $\Omega_{F,s}$ corresponds to a 1-D domain 
that serves as a resolved fracture region. All integrations on $\Omega$
is thus rewritten as $\int_{\Omega_m} +\sum_s\int_{\Omega_{F,s}}$. For example, \eqref{eq:mf_variational} can be rewritten as
\begin{equation}\label{eq:DFN_multi_con}
\begin{split}
	&\int_{\Omega_M}b^i \frac{\partial u^i}{\partial t}v^i\ dx + \sum_s	\int_{\Omega_{F,s}}b_{F,s} \frac{\partial u^i}{\partial t}v^i\ dx\\
	+&\int_{\Omega_M}\kappa^i(x)\alpha(u^i)\nabla u^i\nabla v^i\ dx + \sum_s\int_{\Omega_{F,s}}\kappa_{F,s}(x)\alpha(u^i)\nabla_F u^i\nabla_F v^i\ dx\\
	+&\sum_j\int_{\Omega}Q^{i,j}v^i\ dx = \int_{\Omega}f^iv^i\ dx,\qquad\qquad i = m, f,
\end{split}
\end{equation}
after applying DFM to its original form. Similarly, incorporating DFM in \eqref{eq:v_variational} yields 
\begin{equation}\label{eq:DFN_multi_con_v}
	\int_{\Omega_M}b^v \frac{\partial u^v}{\partial t}v^v\ dx + \sum_s	\int_{\Omega_{F,s}}b_{F,s} \frac{\partial u^v}{\partial t}v^v\ dx - \int_{\Omega}Q^{m,v}v^v\ dx  - \int_{\Omega}Q^{f,v}v^v \ dx= \int_{\Omega}f^vv^v\ dx.
\end{equation}
The fine-scale FEM solution $(u^m, u^f, u^v)$ should be sought in $V_h = V_h^m\times V_h^f\times V_h^v$, where $\{V_h^i \}$, 
where the $V_h^i$ is a conforming finite element space of the continuum $i$ on a fine partition $\mathcal{T}^h$ of domain. 
We also remark that the shape factor $\sigma$ is taken to be proportional to $h^{-2}$. 

For the purpose of simpler notations in the analysis presented in Appendix~\ref{sec:proofs}, 
we rewrite the derived system \eqref{eq:DFN_multi_con}-- \eqref{eq:DFN_multi_con_v} in a more general 
$N$-continuum setting. First, we denote the Sobolev space $V = [H_0^1(\Omega)]^N$. 
On each continuum, given a fixed $w^i \in H_0^1(\Omega)$, we define bilinear forms: 
\begin{equation}
\begin{split}
b^i(u^i,v^i) & = \int_{\Omega_M}b^i u^i v^i \ dx + \sum_s \int_{\Omega_{F,s}}b_{F,s} u^i v^i\ dx, \\
{a}^i(u^i,v^i;w^i) & = \int_{\Omega_M}\kappa^i\alpha(w^i) \nabla u^i\cdot\nabla v^i\ dx + \sum_s \int_{\Omega_{F,s}}\kappa_{F,s}\alpha(w^i) \nabla_F u^i\cdot\nabla_F v^i \ dx.
\end{split}
\end{equation}
Given a fixed $w \in V$, we further define the following coupled bilinear forms on $V$
\begin{equation}
\begin{split}
b(u,v) & = \sum_{i}b^i(u^i,v^i), \\
a(u,v;w) & = \sum_{1 \leq i < N} a^i(u^i, v^i; w^i), \\
q(u,v) & = \sum_{i}\sum_{j\neq i} q^{i,j} \int_{\Omega}(u^i-u^j)v^i\ dx.
\end{split}
\end{equation}
Then the weak formulation \eqref{eq:DFN_multi_con}-- \eqref{eq:DFN_multi_con_v} can be written as: 
find $u = (u^1,u^2, \cdots, u^N)$, where $u(t,\cdot)\in V$,
 such that for all $v = (v^1,v^2, \cdots, v^N)$, where $v(t,\cdot)\in V$,
\begin{equation}\label{eq:weak}
	b\left(\frac{\partial u}{\partial t}, v\right) + a(u, v;u) + q(u,v) = (f,v),\qquad t \in(0,T),
\end{equation}
with $N=3$ continua and the continuum indices representing the matrix, fracture and vug components in order.

%
%
%

\section{GMsFEM}\label{sec_GMsFEM}
 In order to reduce the computational cost, we would like to solve the equation system \eqref{eq:multi-con} and \eqref{eq:vug_eq} on coarse mesh. However, permeability coefficient $\kappa(x)$ is heterogeneous in space, thus a standard FEM solution on coarse mesh will be inaccurate as it loses subgrid information. Therefore, we use GMsFEM to construct multiscale basis that contains local heterogeneous permeability information. By replacing the standard FEM basis with GMsFEM basis, we are able to obtain a better accuracy and sustain an affordable computational cost.
 
 In this section, we briefly review the procedure for GMsFEM. Roughly speaking, the construction of GMsFEM basis consists of two stages: solving snapshot problems and conducting spectral decomposition. Both steps are conducted locally.
 \subsection{Coarse and Fine Mesh} 
We first introduce the coarse grid $\mathcal{T}^H$ with mesh size $H$. And each coarse block in $\mathcal{T}^H$ can be denoted as $K_j$. $\mathcal{T}^H$ can be further refined by an unstructured fine mesh $\mathcal{T}^h$ with mesh size $h\ll H$. See Figure \ref{fig:mesh} for an illustration. We assume $\mathcal{T}^h$ is fine enough to resolve all underlying fine-scale properties of $\kappa(x)$. Let $\{x_i| 1\leq i\leq N_v\}$ be the set of all coarse nodes of $\mathcal{T}^H$, where $N_v$ is the total number of coarse nodes. We then define the coarse neighborhood $\omega_i$ of node $x_i$ as $$\omega_i = \bigcup_j \{K_j|x_i\in K_j\}.$$

\begin{figure}
	\centering
	\includegraphics[width = \textwidth]{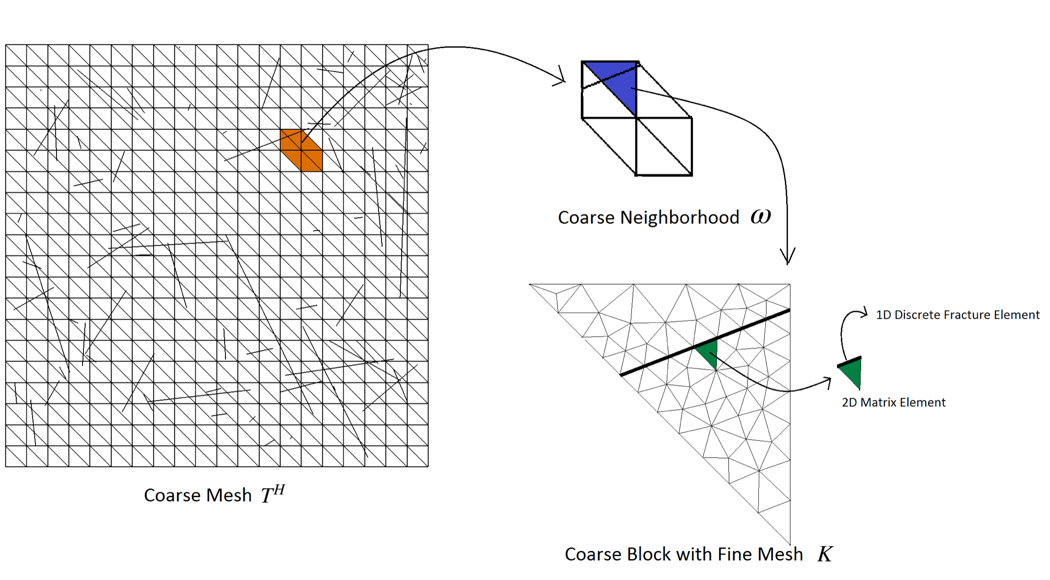}
	\caption{Coarse \& Fine Mesh. Left: coarse mesh with discrete fractures; Upper Right: A coarse neighborhood; Lower Right: A coarse block with a discrete fracture and refined mesh.}
	\label{fig:mesh}
\end{figure}
 
 \subsection{Snapshot Space}
 A snapshot space is an auxiliary space constructed within each coarse neighborhood $w_i$. We omit the subscript $i$ for simplicity. There are a few different ways of constructing snapshot space \cite{efendiev2013generalized}. In this paper, we take solutions to the following harmonic extension problems as snapshot basis functions of three coninuum. The snapshot space is exactly the span of all such basis functions.
 
The snapshot problems are designed analogue to the steady state equation of \eqref{eq:multi-con} and \eqref{eq:vug_eq}. 
We consider a \emph{coupled} snapshot system in a coarse neighborhood $\omega$, in which 
we find $\phi_{k,s}^{\text{snap},\omega} = \left(\phi_{k,s}^{m,\text{snap},\omega}, 
\phi_{k,s}^{f,\text{snap},\omega}, \phi_{k,s}^{v,\text{snap},\omega}\right) \in V_h$ such that 
 \begin{equation}\label{eq:snap_coupled}
 \begin{split}
-\nabla\cdot{\frac{\kappa^i(x)}{\mu}\nabla\phi^{i,\text{snap},\omega}_{k,s}} +\sum_{j \neq i} q^{i,j}(\phi^{i,\text{snap},\omega}_{k,s} - \phi^{j,\text{snap},\omega}_{k,s})&=0  \qquad \text{in } \omega \quad i = m,f, \\
\sum_{j \neq v} q^{v,j}(\phi^{v,\text{snap},\omega}_{k,s} - \phi^{j,\text{snap},\omega}_{k,s})  &= 0 \qquad \text{in } \omega,\\
{\phi}^{\text{snap},\omega}_{k,s} &= {\delta}_{k,s} 
\quad  \text{ on } \partial\omega.	
 \end{split}
 \end{equation}
${\delta}_{k,s}$ is defined on all fine-scale nodes of $\partial\omega$. 
If the set $\{x^\omega_i| 1\leq i \leq N^{\omega}_v\}$ represents all fine-scale nodes on boundary, we have 
 $$\delta_{k,s}(x^{\omega}_i) = \left\{
 \begin{array}{ll}
 {e}_s      &  \quad i=k,\\
 {0}      & \quad i\neq k.
 \end{array}
 \right. $$ 
 Here, $\{{e}_s\}_{s=1}^3$ is standard basis in $\mathbb{R}^3$.
 So far, we have constructed the local snapshot space as:
$$V_{\text{snap}}^{\omega} = \text{span}\{{\phi}^{\text{snap},\omega}_{k,s} \ |\ 1\leq k\leq N^{\omega}_v, 1 \leq s\leq 3\}.$$
The global snapshot space is defined as the sum of all local snapshot spaces, i.e. 
$$V_{\text{snap}}= \text{span}\{{\phi}^{\text{snap},\omega_i}_{k,s} \ |\ 1 \leq i \leq N_v, 1\leq k\leq N^{\omega_i}_v, 1 \leq s\leq 3\}.$$
 \emph{Remark } When solving local snapshot problem \eqref{eq:snap_coupled} on the fine mesh within $\omega$, one should also apply the idea of DFM and replace all integral $\int_{\omega}$ by $\int_{\omega_M}+\sum_s\int_{\omega_{F,s}}$ and all coefficient correspondingly.
 \subsection{Spectral Problem}
To further reduce the dimension of resulting system, we conduct a spectral decomposition on $V_{\text{snap}}^{\omega}$. Such decomposition will automatically detect the dominant modes. More precisely, 
we sought eigenpairs $(\lambda^\omega_k,\psi^\omega_k) \in \mathbb{R} \times V_{\text{snap}}^{\omega}$ 
for the following local spectral problem
\begin{equation}\label{eq:spectral}
	a_\omega(\psi^\omega_k, v) = \lambda^\omega_k s_\omega(\psi^\omega_k,v) \qquad \forall v \in  V_{\text{snap}}^{\omega}, 
\end{equation}
where
\begin{equation*}
\begin{split}
a_\omega(u,v) & = \sum_{i \in \{m,f\}}a_\omega^i(u^i,v^i) +\sum_{i}\sum_{j\neq i}\int_{\omega}q^{i,j}(u^i-u^j)v^i\ dx,\\
s_\omega(u,v) & =\frac{1}{\mu}\sum_{i}\int_{\omega}\kappa^i(x) u^i v^i\ dx.
\end{split}
\end{equation*}
The form of $a_\omega(u,v)$ and $s_\omega(u,v)$ are inspired by analysis which will 
be demonstrated in next section along with Appendix~\ref{sec:proofs}.
We sort the eigenvalues $\{\lambda^\omega_k\}$ of \eqref{eq:spectral} in ascending order, 
and we take the first $L_\omega$ eigenfunctions $\psi^\omega_k = (\psi^{m,\omega}_k, 
\psi^{f,\omega}_k, \psi^{v,\omega}_k)$.
Then the $k$-{th} multiscale basis function
$\psi^{\omega}_{k,ms} = (\psi^{m,\omega}_{k,ms},\psi^{f,\omega}_{k,ms}, \psi^{m,\omega}_{k,ms})$ in $\omega$ is defined by
$$\psi^{i,\omega}_{k,ms} = \chi^{\omega}\psi^{i,\omega}_k, \quad i = m,f,v,$$
where $\chi^{\omega}$ is a partition of unity function for coarse grid $\mathcal{T}^H$ on a coarse neighborhood $\omega$. 
By multiplying $\chi^{\omega}$, we obtained a set of conforming multiscale basis functions supported in $\omega$. 
Using the multiscale basis functions $\{\psi^{\omega_i}_{k,ms}\}$ for all coarse regions $\omega_i$, 
we construct the multiscale space
$$V_{\text{ms}} = \text{span}\{\psi^{\omega_i}_{k,ms} \ |\  1\leq i\leq N_v,1\leq k\leq L_{\omega_i}\} .$$ 
We remark that $\text{dim } V_{\text{ms}} \ll \text{dim } V_h$, where $V_h = V_h^m \times V_h^f \times V_h^v$ is the standard FEM approximation space on $\mathcal{T}^h$. When the multiscale space is established, we can find a coarse-scale solution on $V_{\text{ms}}$ with less computational effort.

Once the multscale space is constructed, 
the GMsFEM solution is given by: find $u_{\text{ms}}= (u_{\text{ms}}^1,u_{\text{ms}}^2, \cdots, u_{\text{ms}}^N)$, 
where $u_{\text{ms}}(t,\cdot) \in V_{\text{ms}}$, 
such that for all $v = (v^1,v^2, \cdots, v^N)$, where $v(t,\cdot)\in V_{\text{ms}}$,
\begin{equation}\label{eq:ms}
	b\left(\frac{\partial u_{\text{ms}}}{\partial t}, v\right) +  a(u_{\text{ms}}, v;u_{\text{ms}}) + q(u_{\text{ms}},v) = (f,v),\qquad t \in(0,T).
\end{equation}

\subsection{A-priori error estimates}
In this section, we present some a-priori error estimates 
of the semi-discrete problem. 
The proofs of these estimates 
will be left to Appendix~\ref{sec:proofs}.

We suppose the field $\kappa$ has a upper bound $\kappa^+$ and a lower bound $\kappa^-$ on $\Omega$.
We further assume that the fields $\alpha(u^i)$ and $\alpha(u^i_{\text{ms}})$ 
has a uniform upper bound $\alpha^+$ and a uniform lower bound $\alpha^-$, i.e.
\begin{equation}\label{eq:a_bound}
0< \alpha^- \leq \alpha(u^i), \alpha(u^i_{\text{ms}}) \leq \alpha^+.
\end{equation}
Next, we introduce some metrics on $V$. 
The bilinear form $b(\cdot,\cdot)$ can further induce a norm
$$\|u\|_b = (b(u,u))^{1/2}.$$
We also define a norm $\|\cdot\|_{a_Q}$ by
\begin{equation}
\| u \|_{a_Q} = ( \vert u \vert_a^2 + \vert u \vert_q^2 )^\frac{1}{2}, 
\end{equation}
where
\begin{equation}
\begin{split}
\vert u \vert_a^2 & = \sum_{1 \leq i < N} \left( \int_{\Omega_M}\kappa^i \vert \nabla u^i\vert^2 \ dx + 
\sum_s \int_{\Omega_{F,s}}\kappa_{F,s} \vert \nabla_F u^i\vert^2 \ dx\right),  \\
\vert u \vert_q^2 & = q(u,u).
\end{split}
\end{equation}

The first theorem provides an estimate of 
the error between the weak solution $u$ and the multiscale solution $u_{\text{ms}}$ 
by the projection error of $u$ onto the multiscale space $V_{\text{ms}}$ in various metrics.
\begin{theorem}
	Let $u$ be the weak solution in \eqref{eq:weak} and 
	$u_{\text{ms}}$ be the multiscale numerical solution in \eqref{eq:ms}. 
	Assume $\nabla u \in L^4(\Omega_M)$ and $\nabla_F u \in L^\infty(\Omega_{F,s})$. 
	Then we have 
	\begin{equation}\label{eq:lemma1}
	\begin{split}
	&\|u(t,\cdot) - u_{\text{ms}}(t,\cdot)\|_b^2 +\int_0^T\|u-u_{\text{ms}}\|_{a_Q}^{2}\ dt \\
	&\leq\ C \inf_{w\in V_{\text{ms}}}(\int_0^T \|\frac{\partial(w-u)}{\partial t}\|_b^2\ dt + \int_0^T\|w-u\|_{a_Q}^2\ dt + \|w(0, \cdot) - u(0,\cdot)\|_b^{2}).
	\end{split}
	\end{equation}
\label{thm1}
\end{theorem}

In light of Theorem~\ref{thm1}, we have to establish an estimate of the projection error of $u$ 
onto the multiscale space $V_{\text{ms}}$ in various metrics on the right hand side of \eqref{eq:lemma1}, 
in order to complete the convergence analysis. 
With the assumption that the irreducible error between the Sobolev space $V$ 
and the snapshot space $V_\text{snap}$ is small, which holds when a sufficiently large number of 
snapshot solutions is taken, 
we define an approximation $u_{\text{snap}}(\cdot,t) \in V_{\text{snap}}$ of $u(\cdot,t)$ in the snapshot space by
\begin{equation}
\label{eq:u_snap}
u_{\text{snap}}(x,t) = \sum_{i=1}^{N_v} \sum_{k=1}^{N_v^{\omega_i}} \sum_{s=1}^3 
u(x_k,t) \chi^{\omega_i}(x_k) \phi^{\text{snap},\omega_i}_{k,s}(x),
\end{equation} 
and provide an estimate of the projection error of $u_{\text{snap}}$ onto the snapshot space $V_{\text{ms}}$. 

\begin{theorem}\label{thm2}
	Let $u$ and $u_{\text{snap}}$ be reference solution and snapshot projection of $u$ as defined in \eqref{eq:weak} and \eqref{eq:u_snap}. 
	Then we have
	\begin{equation}
	\begin{split}
	\inf_{w \in V_{\text{ms}}} \int_0^T\|\frac{\partial(w-u_{\text{snap}})}{\partial t}\|_b^2\ dt& + \int_0^T\|w-u_{\text{snap}}\|_{a_Q}^2\ dt +\|w(0,\cdot)-u_{\text{snap}}(0,\cdot)\|_b^2\\
	&\leq\frac{C}{\Lambda} (\int_0^T\|\frac{\partial u}{\partial t}\|_{a_Q}^2\ dt + \int_0^T\|u\|_{a_Q}^2\ dt + \|u(0,\cdot)\|_{a_Q}^2)
	\end{split}
	\end{equation} 
	with
	$$\Lambda = \min_{j}\{\lambda^{\omega_j}_{L_{\omega_j}+1}\}.$$
\end{theorem}

\subsection{An implementation view}
In this section, we derive the fully discrete system 
and present the implementation details. 
We adopt the implicit Euler scheme for time discretization to the semi-discrete 
GMsFEM system \eqref{eq:ms}. 
Suppose the time domain $(0,T)$ is partitioned into equal subintervals of length $\Delta t$, 
and denote the $n$-th time instant by $t_n = n\Delta t$. 
Using backward difference, the fully discrete GMsFEM scheme 
is to, successively for $n = 1,2,\ldots,$ find $u_{\text{ms}}^{n} \in V_{\text{ms}}$ such that 
\begin{equation}\label{eq:full_ms}
	b\left(\frac{u^n_{\text{ms}}-u^{n-1}_{\text{ms}}}{\Delta t}, v\right) +  
	a(u^n_{\text{ms}}, v;u^n_{\text{ms}}) + q(u^n_{\text{ms}},v) = (f^n,v) \text{ for all } v \in V_{\text{ms}},
\end{equation}
where the subscript $n$ denotes the evaluation of a time-dependent function at 
the time instant $t_n$ and an initial condition $u_{\text{ms}}^0$ is given. 
At each time instant $t_n$, \eqref{eq:full_ms} gives rise to a nonlinear algebraic system in 
the coefficients with respect to the multiscale basis functions. 
With a sufficiently small time step size, we can adopt a direct linearization approach 
by replacing the field $\alpha(u^n_{\text{ms}})$ by $\alpha(u^{n-1}_{\text{ms}})$ and derive
\begin{equation}\label{eq:linearization}
	b\left(\frac{u^n_{\text{ms}}-u^{n-1}_{\text{ms}}}{\Delta t}, v\right) +  
	a(u^n_{\text{ms}}, v;u^{n-1}_{\text{ms}}) + q(u^n_{\text{ms}},v) = (f^n,v) \text{ for all } v \in V_{\text{ms}}.
\end{equation}
Alternatively, we can use an iterative approach. More precisely, 
we can construct a sequence $\{u^n_{\text{ms},m}\}_{m=0}^\infty \subset V_{\text{ms}}$ 
whose fixed point is the solution $u^n_{\text{ms}}$ and 
truncate the successive iterations when a stopping criterion is fulfilled. 
In this case, we start with an initial guess $u^n_{\text{ms},0} = u^{n-1}_{\text{ms}}$ 
and solve for 
\begin{equation}\label{eq:iteration}
	b\left(\frac{u^n_{\text{ms},m}-u^{n-1}_{\text{ms}}}{\Delta t}, v\right) +  
	a(u^n_{\text{ms},m}, v;u^n_{\text{ms},m-1}) + q(u^n_{\text{ms},m},v) = (f^n,v) \text{ for all } v \in V_{\text{ms}}.
\end{equation}
We remark that it is equivalent to the linearization approach if we stop after one iteration.

\section{Numerical Results}\label{sec_numerical}
In this section, we apply our proposed methods to a realistic fractured and vuggy reservoir. All three continuum have heterogeneous permeability background (see Figure \ref{fig:ms_perm} for the permeability of matrix) and discrete fracture networks are embeded in this reservoir like in Figure \ref{fig:fracture}. An unstructured fine mesh is used to resolve the discrete fractures networks(see Figure \ref{fig:fine_mesh}). The descriptive parameters of this reservoir are listed in Figure \ref{fig:ms_perm} and Table \ref{tab:values}. All numerical results are implemented using FEniCS Library.

%
\begin{figure}[h!]
	\centering
	\includegraphics[width = 0.8\textwidth]{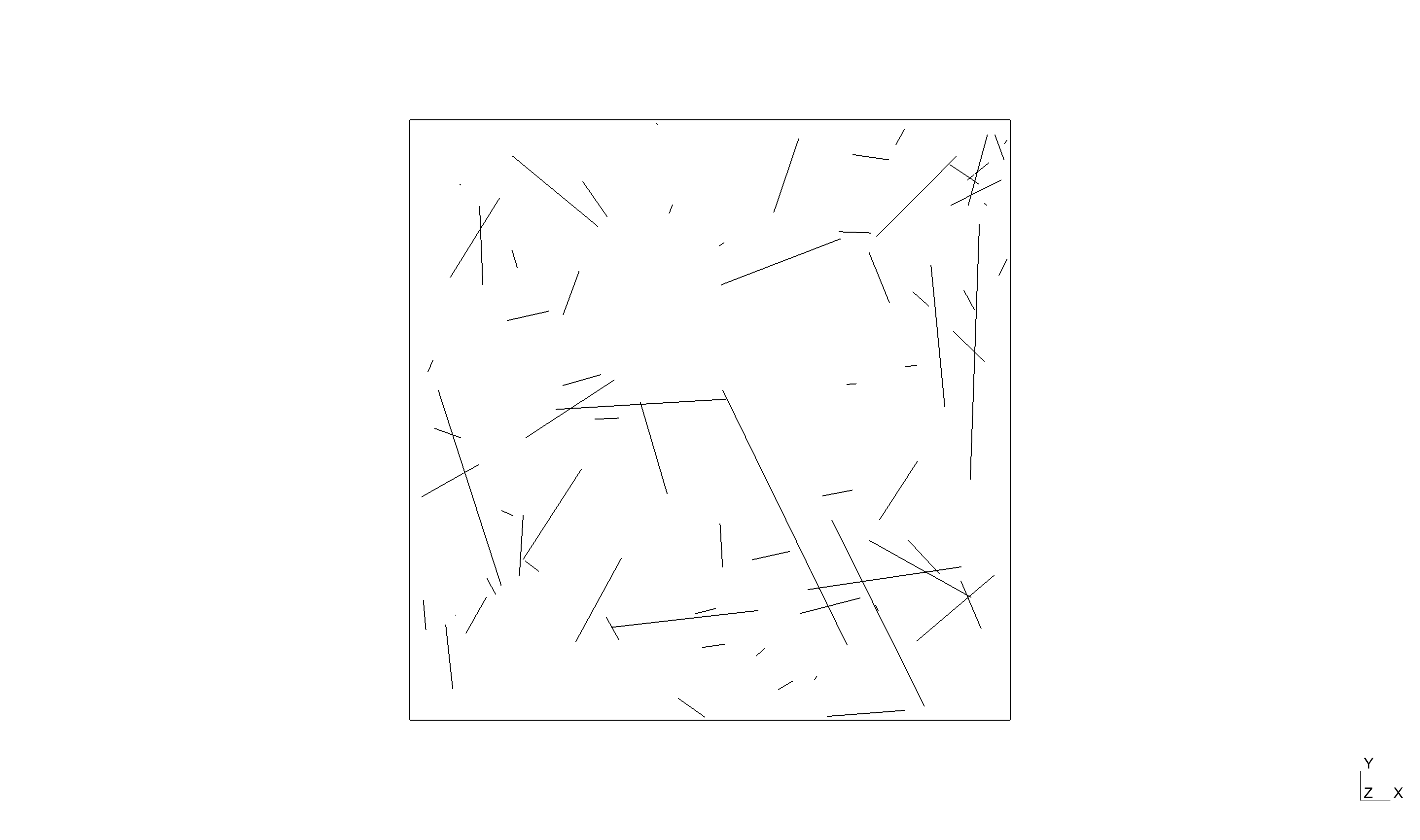}
	\caption{Idealized discrete fracture network(DFN)}
	\label{fig:fracture}
\end{figure}

\begin{figure}[h!]
	\centering
	\includegraphics[width = 0.6\textwidth]{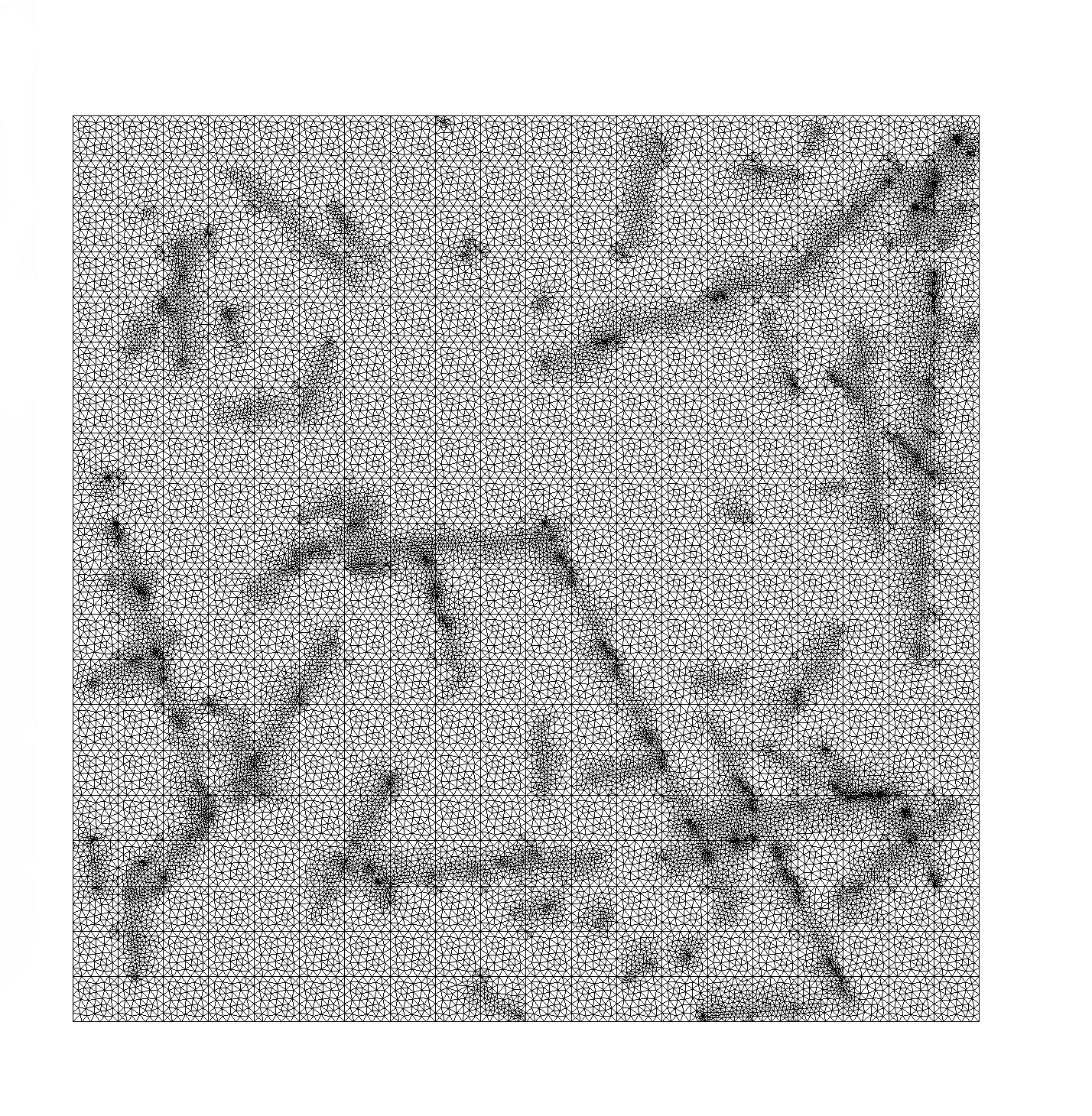}
	\caption{Unstructured fine mesh}
	\label{fig:fine_mesh}
\end{figure}

Numerical experiments are conducted from different aspects. Performance are compared between MsFEM and GMsFEM, nonzero source term and nonzero mixed boundary condition. We also discuss the impact of the number of basis function selected to the solution accuracy. We remark that all examples are conduced using direct linearization approch as the iterative approach do not significantly improve the results for our problem, which indicates that the nonlinearity in our problem is not very strong.
\subsection{Comparison of MsFEM and GMsFEM}
In this subsection, we discuss the necessity to apply GMsFEM. From Figure \ref{fig:Ms_GMs}, we can tell that, even with similar number of degrees of freedom, the MsFEM is not able to resolve the true solution, thus GMsFEM must be applied to generate meaningful results. This is especially true when there are multiple discrete fracture networks coexist in a single coarse neighborhood. Many numerical experiments have shown that MsFEM basis functions are not able to handle homogeneous background and multiple discrete fracture networks simultaneously.
Figure \ref{fig:Ms_GMs} shows the solution we obtained using MsFEM and GMsFEM respectively when a single source is placed at the bottom left corner. The error of MsFEM solution can be as large as $30\%$.
\begin{figure}[h!]
	\centering
	\includegraphics[width = 0.7\textwidth]{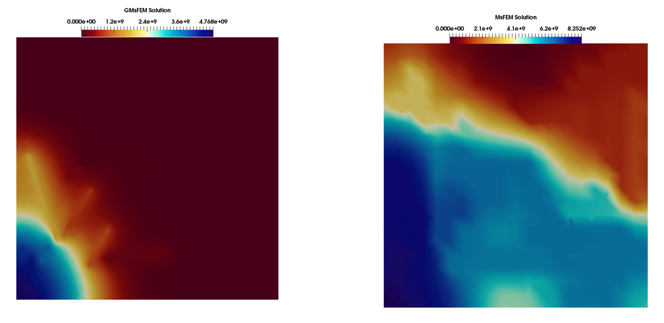}
	\caption{Comparison between GMsFEM and MsFEM solution with heterogeneous background and discrete fracture network. Left: GMsFEM solution with DOF=2646. Right MsFEM solution with DOF =2400.}
	\label{fig:Ms_GMs}
\end{figure}

\subsection{GMsFEM solution for different boundary condition and source}
In this subsection, we demonstrate the performance of our proposed triple continuum GMsFEM solution to problem \eqref{eq:multi-con} and \eqref{eq:vug_eq} ,where lagging coefficient scheme is used to linearize the problem.

Different boundary condition and source term settings are tested for coupled GMsFEM approach.

\begin{figure}[h!]
	\centering
	\includegraphics[width = \textwidth]{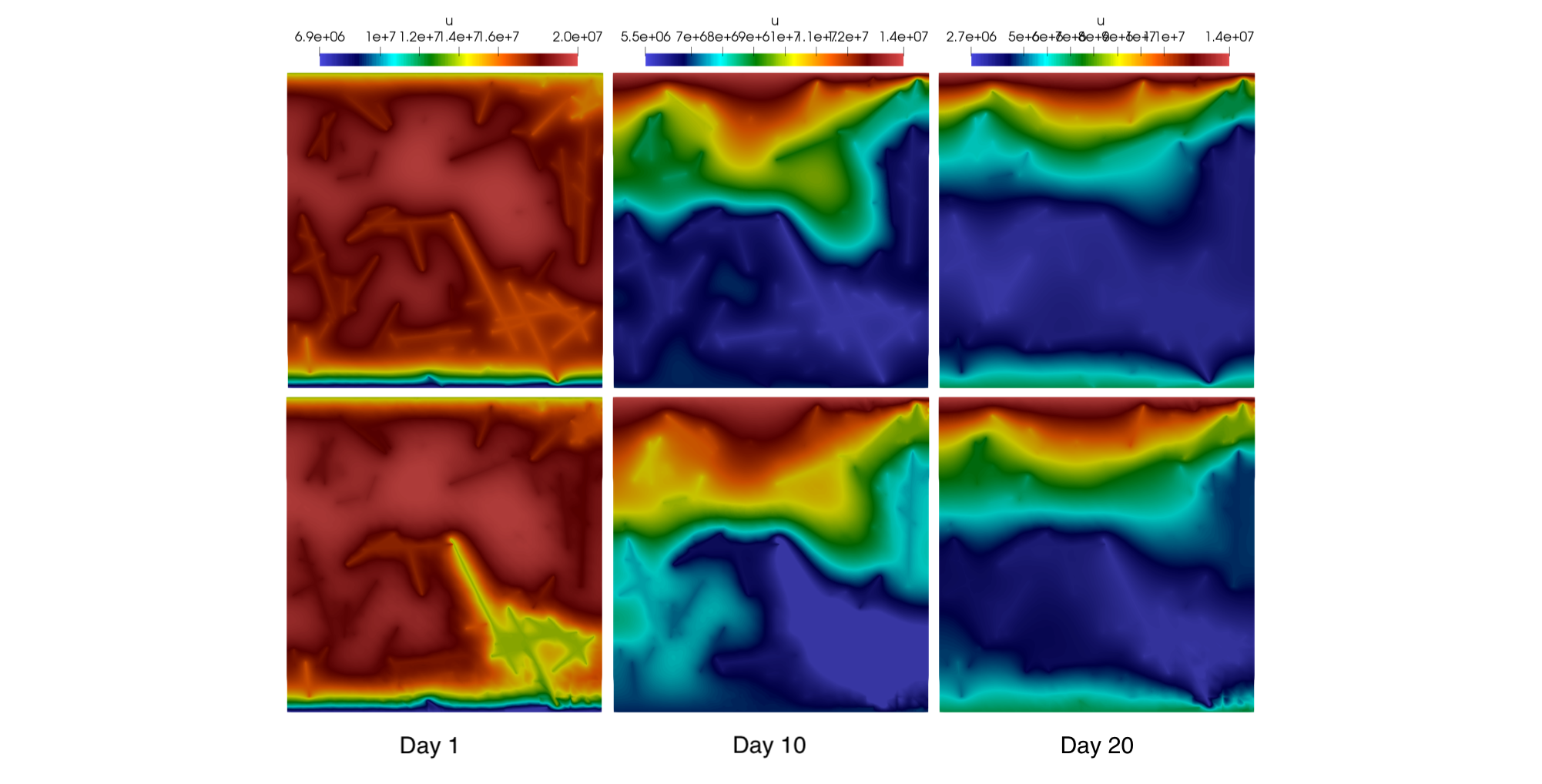}
	\caption{Triple-Continuum, heterogeneous background flow simulation matrix with top and bottom nonzero Dirichlet boundary condition. Zero Neumann boundary is applied to left and right boundary. First row: fine-scale reference solution, DOF = 80229. Second row: Coupled coarse-scale GMsFEM solution with 8 basis, DOF = 3528. }
	\label{fig:dirichlet}
\end{figure}

\begin{table}[h!]
  \centering
    \begin{tabular}{|c|r|r|r|}
    	\hline
     Number of Basis     & Day 1     & Day 10    & Day 20 \\
         	\hline
    2     & 17.21 & 27.22 & 66.44 \\
    4     & 14.88 & 17.27 & 43.65 \\
    8     & 4.72  & 11.86 & 13.31 \\
    16   & 4.24  & 12.05 & 12.58\\
        	\hline
    \end{tabular}%
    \caption{$L^2$ relative errors(\%) of numerical results for mixed boundary condition. Nonzero Dirichlet boundary condition is imposed on top and bottom boundary. Zero Neumann boundary is applied to left and right boundary.}
      \label{tab:dirichlet}
\end{table}%

\begin{figure}[h!]
	\centering
	\includegraphics[width = 0.8\textwidth]{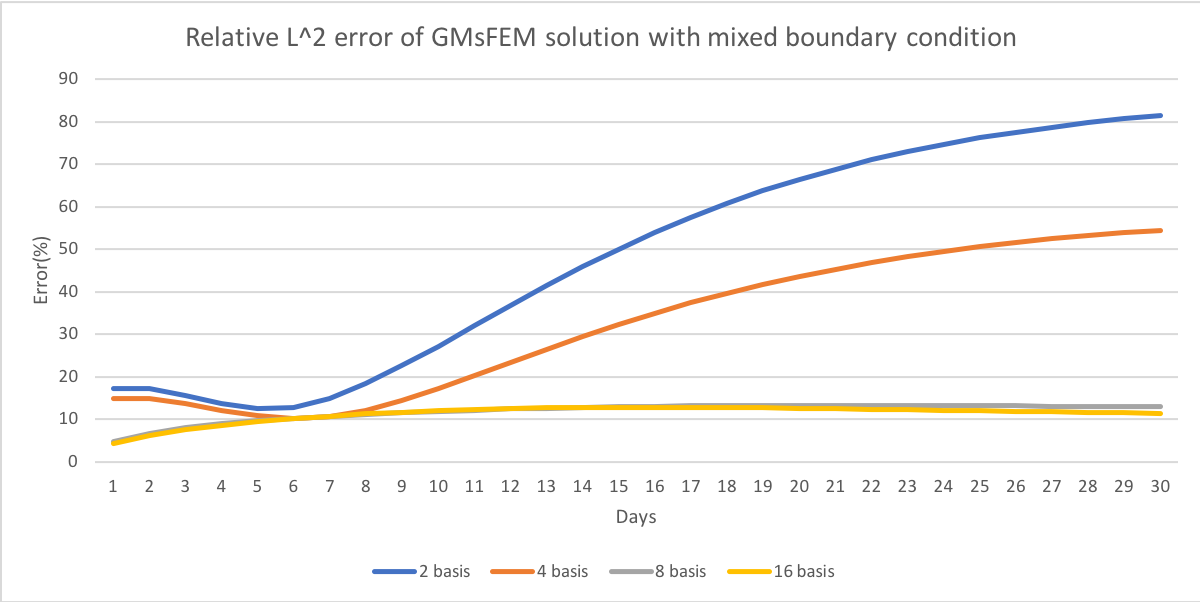}
	\caption{ Illustration of error trend with time for different number of basis for dirichlet boundary condition case}
	\label{fig:diff_eigen_trend_diri}
\end{figure}

%

\begin{figure}[h!]
	\centering
	\includegraphics[width = \textwidth]{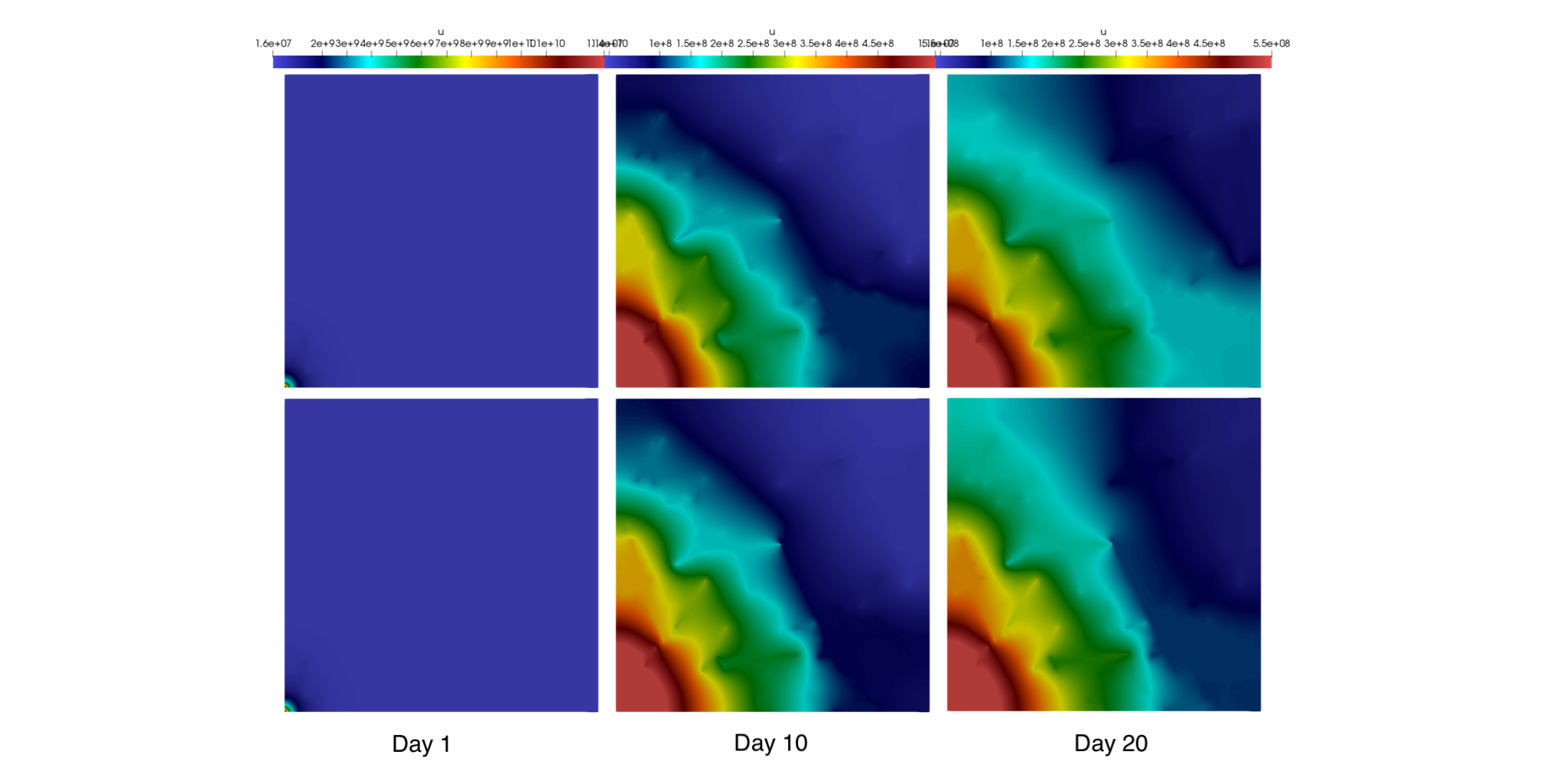}\label{fig:source}
	\caption{Flow simulation results for a triple continuum heterogeneous background matrix with no flow boundary condition. Injector located at bottom left corner. 
	First row : Fine-scale reference solution. Second row: Coupled coarse-scale GMsFEM solution. DOF is same as in Figure \ref{fig:dirichlet}.}
\end{figure}

\begin{table}[h!]
  \centering
    \begin{tabular}{|c|r|r|r|}
    	\hline
       Number of Basis      &Day 1     & Day 10    & Day 20 \\
       \hline
    2     & 15.79 & 10.05 & 11.42 \\
    4     & 5.48  & 5.89  & 8.53 \\
    8     & 2.84  & 6.20  & 8.51 \\
    16     & 1.12 & 6.30  & 8.49\\
    \hline
    \end{tabular}%
 \caption{$L^2$ relative errors(\%) of numerical results for zero Neumann boundary condition.}
  \label{tab:source}%
\end{table}%

\begin{figure}[h!]
	\centering
	\includegraphics[width = 0.8\textwidth]{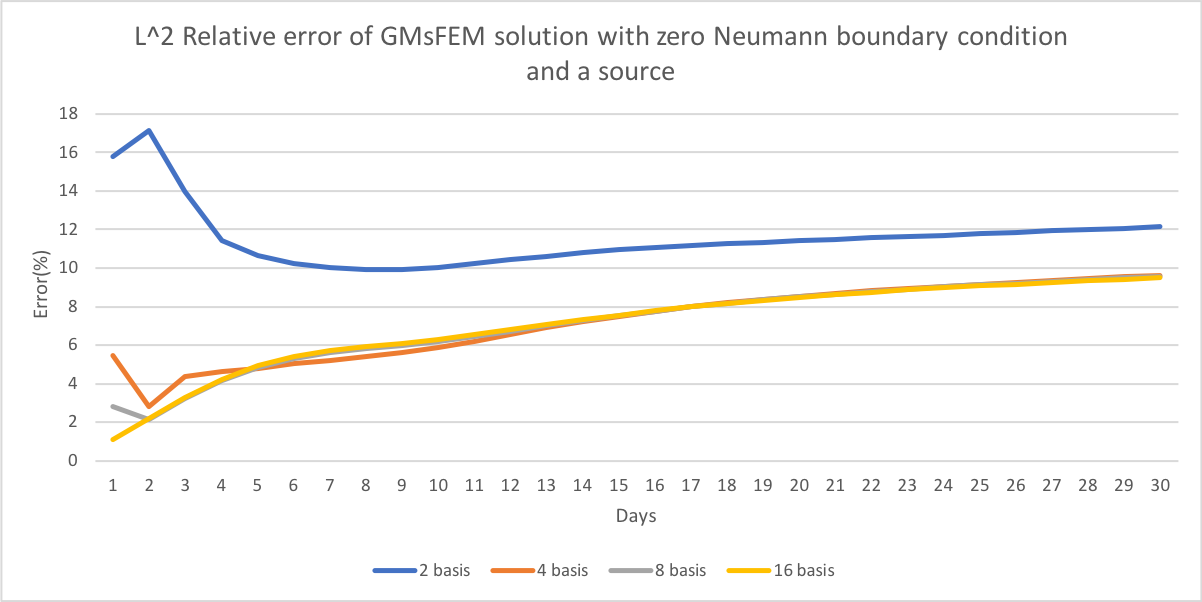}
	\caption{ Illustration of error trend with time for different number of basis for single source}
	\label{fig:diff_eigen_trend_source}
\end{figure}

%
%

From both error tables and solution figures , we come to 
the conclusion that: 1) For nonzero mixed boundary condition case, the GMsFEM solution can obtain a good result when using 8 basis or more.
 2) For zero Neumann boundary and single point source term case, the coupled approach can obtain good approximation of fine-scale solution with 4 basis or more. 3) For both cases, the coupled approach can give us an acceptable solution. 

%

From Figure \ref{fig:diff_eigen_trend_source} , Figure \ref{fig:diff_eigen_trend_diri}, Table \ref{tab:source} and Table \ref{tab:dirichlet}, we can tell that the error of solution 
decrease when we increase the number of eigen-functions used.

\section{Conclusion}\label{sec_conclusion}
In this paper, we proposed a triple continuum GMsFEM method as a fast solver of flow problems in heterogeneous domain. A fractured and vuggy reservoir is modeled as a coupled system of three continuum. Fractures are treated hierarchically, fractures with only global effects are considered as a continua, while the ones have local effects are represented as discrete fracture networks using DFM. The system coupling DFM and three continuum are discretized spatially following the Generalized Multiscale Finite Element Method(GMsFEM) for accurate and fast solution. Coupled assembling is provided to construct GMsFEM multiscale space.
The convergence of our proposed method is proved strictly following mild 
assumptions. Later, the performance is tested using multiple examples with 
different settings. We conclude that GMsFEM is necessary for complicated 
discrete fracture networks, the proposed approach can provide 
competitive approximation for both mixed boundary conditions and a single source case. The number of basis are also discussed and chosen. From the numerical exmaples, we can see that selecting enough number of basis is crutial to the accuracy of our proposed method.

 In short, we claim that our proposed method can accomplish the 
flow simulation task with both accuracy and efficiency. Nevertheless, we 
notice that our proposed method is only good for the case when a clear 
description of the discrete fracture networks is known. For reservoirs containing 
uncertainties, further exploration is desired. Besides, for vugs with turbulent 
flow inside, one will end up with a coupled PDE system containing Navier-Stokes 
system. Future investigations are required to expand our work to such cases.

\appendix\section{Proofs of error estimates}
\label{sec:proofs}

In this section, we present the proofs of the 
error estimates in Theorem~\ref{thm1} 
and Theorem~\ref{thm2}. 

\subsection{Proof of Theorem~\ref{thm1}}


\begin{proof}
Using \eqref{eq:weak} and \eqref{eq:ms}, we have 
$$b(\frac{\partial(u-u_{\text{ms}})}{\partial t}, v) + \sum_{1 \leq i < N} ({a}^i(u_i,v^i;u_i) - {a}^i(u_{\text{ms}}^i,v^i;u_\text{ms}^i)) + q(u-u_{\text{ms}}, v) = 0 \qquad\forall v\in V_{\text{ms}},\ t\in (0,T).$$
Let $w \in V_{\text{ms}}$ and take $v=w -u_{\text{ms}}$, we have 
\begin{equation*}
\begin{split}
&b(\frac{\partial (w-u_{\text{ms}})}{\partial t}, w -u_{\text{ms}})+ q(w-u_{\text{ms}}, w -u_{\text{ms}}) -  \sum_{1 \leq i < N} {a}^i(u^i_{\text{ms}} ,w^i-u_{\text{ms}}^i;u^i_{\text{ms}})\\
 =\ &b(\frac{\partial (w-u)}{\partial t}, w -u_{\text{ms}}) + q(w-u, w -u_{\text{ms}})-  \sum_{1 \leq i < N} {a}^i(u^i ,w^i-u_{\text{ms}}^i;u^i)
\end{split}
\end{equation*}
From this equation ,we can further get the following by the definition of ${a}^i$
and the bounded condition \eqref{eq:a_bound} of $a(u)$,
\begin{equation*}
\begin{split}
&b(\frac{\partial (w-u_{\text{ms}})}{\partial t}, w -u_{\text{ms}})+ 
\alpha^-q(w-u_{\text{ms}}, w -u_{\text{ms}}) +\alpha^- 
|w-u_{\text{ms}}^i|_a^2\\
 \leq\ & |b(\frac{\partial (w-u)}{\partial t}, w -u_{\text{ms}})| + 
\alpha^+|q(w-u, w -u_{\text{ms}})|+\alpha^+|w-u|_a |w-u_{\text{ms}}|_a\\
 &+ \sum_{1 \leq i < N} \int_{\Omega}|(\alpha(u_{\text{ms}}^i)-a(u_i))\kappa^i\nabla u_i\cdot \nabla(w^i-u_{\text{ms}}^i)|\ dx
\end{split}
\end{equation*}
By Cauchy-Schwarz Inequality, this implies
\begin{equation}\label{eq:mid2}
\begin{split}
&\frac{1}{2}\frac{d}{dt}\|w-u_{\text{ms}}\|_b^2 + \alpha^-\|w-u_{\text{ms}}\|^2_{a_Q}\\
 \leq & \|\frac{\partial (w-u)}{\partial t}\|_b\|w-u_{\text{ms}}\|_b +\alpha^+\|w-u\|_{a_Q}\|w-u_{\text{ms}}\|_{a_Q}\\
&+ \sum_{1 \leq i < N} \int_{\Omega}|(\alpha(u_{\text{ms}}^i)-\alpha(u^i))\kappa^i\nabla u^i\cdot \nabla(w^i-u_{\text{ms}}^i)|\ dx
\end{split}
\end{equation}
The last term on the right-hand side of \eqref{eq:mid2} can be written as 
\begin{equation}\label{eqqqq}
\begin{split}
& \int_{\Omega}|(\alpha(u_{\text{ms}}^i)-\alpha(u^i))\kappa^i\nabla u^i\cdot\nabla(w^i-u_{\text{ms}}^i)|\ dx \\
 =&\int_{\Omega_M}|(\alpha(u_{\text{ms}}^i)-\alpha(u^i))\kappa^i\nabla u^i\cdot\nabla(w^i-u_{\text{ms}}^i)|\ dx \\
&+\sum_s\int_{\Omega_{F,s}}|(\alpha(u_{\text{ms}}^i)-\alpha(u^i))\kappa_{F,s}\nabla_F u^i\cdot\nabla_F(w^i-u_{\text{ms}}^i)|\ dx
\end{split}
\end{equation}
Following \cite{farago1991finite}, we employ generalized Holder's Inequality and the definition of $\alpha(\cdot)$ to obtain 
\begin{equation}
\begin{split}
&\int_{\Omega_M}|(\alpha(u_{\text{ms}}^i)-\alpha(u^i))\kappa^i\nabla u^i\cdot\nabla(w^i-u_{\text{ms}}^i)|\ dx \\
\leq\ & \|\alpha(u_{\text{ms}}^i) -\alpha(u^i)  \|_{L^4(\Omega_{M})} \|(\kappa^i)^{1/2} \nabla u^i\|_{L^4(\Omega_M)} \|(\kappa^i)^{1/2} \nabla (w^i - u_{\text{ms}}^i)\|_{L^2(\Omega_M)} \\
=\ &\frac{c}{\mu} \|u_{\text{ms}}^i- u^i\|_{L^4(\Omega_{M})} \|(\kappa^i)^{1/2} \nabla u^i\|_{L^4(\Omega_M)} \|(\kappa^i)^{1/2} \nabla (w^i - u_{\text{ms}}^i)\|_{L^2(\Omega_M)} \\
\end{split}
\end{equation}
Further, with Ladyzhenskaya's Inequality, there exists some constant $C_1>0$ such that
\begin{equation}
 \|u_{\text{ms}}^i- u^i\|_{L^4(\Omega_M)} \leq C_1  \|u_{\text{ms}}^i- u^i\|_{L^2(\Omega_M)}^{1/2} \|\nabla(u_{\text{ms}}^i- u^i)\|_{L^2(\Omega_M)}^{1/2}
\end{equation}
There also exist some constant $K_1,K_2$ such that
\begin{equation*}
\begin{split}
\|\nabla(u^i_{\text{ms}}-u^i)\|_{L^2(\Omega_M)}^2 & = \int_{\Omega_M}(\nabla(u^i_{\text{ms}}-u^i))^2\ dx\leq K_1\int_{\Omega_M}\frac{\kappa^i}{\mu}(\nabla(u^i_{\text{ms}}-u^i))^2\ dx, \\
\|u^i_{\text{ms}}-u^i\|_{L^2(\Omega_M)}^2 & = \int_{\Omega_M}(u^i_{\text{ms}}-u^i)^2\ dx\leq K_2\int_{\Omega_M}b^i(u^i_{\text{ms}}-u^i)^2\ dx.
\end{split}
\end{equation*}

For the fracture part, we have 
\begin{equation}
\begin{split}
&\int_{\Omega_{F,s}}|(\alpha(u_{\text{ms}}^i)-\alpha(u^i))\kappa_{F,s}\nabla_F u^i\cdot\nabla_F(w^i-u_{\text{ms}}^i)|\ dx \\
\leq\ &C_2\|(\kappa_{F,s})^{1/2}\nabla u^i\| _{L^\infty}  \|u_{\text{ms}}^i- u^i\|_{L^2(\Omega_{F,s})} 
\|(\kappa_{F,s})^{1/2} \nabla(w^i-u_{\text{ms}})\|_{L^2(\Omega_{F,s})} \\
 \end{split}
\end{equation}

To sum up, we have for any $\zeta>0$,
\begin{equation}
\begin{split}
&\int_{\Omega}|(\alpha(u_{\text{ms}}^i)-\alpha(u^i))\kappa^i\nabla u^i\cdot\nabla(w^i-u_{\text{ms}}^i)|\ dx\\
\leq\ 
&C_3(\frac{1}{2\zeta}\|u^i_{\text{ms}}-u^i\|_{b}+\frac{\zeta}{2}|u^i_{\text{ms
}}-u^i|_{a})\cdot |w^i-u_{\text{ms}}^i|_{a}
\end{split}
\end{equation}
for some constant $C_3$.
Plug back to \eqref{eq:mid2}, and notice that $\vert \cdot \vert_a\leq \|\cdot\|_{a_Q}$ we can use Young's Inequality to derive
\begin{equation*}
\begin{split}
&\frac{1}{2}\frac{d}{dt}\|w-u_{\text{ms}}\|_b^2 + \alpha^-\|w-u_{\text{ms}}\|^2_{a_Q}\\
\leq\ &\frac{1}{2\eta}\|\frac{\partial (w-u)}{\partial t}\|_b^2 + \frac{\eta}{2}\|w-u_{\text{ms}}\|_b^2 +\frac{\alpha^+}{2\xi}\|w-u\|_{a_Q}^2 + \frac{\alpha^+\xi}{2}\|w-u_{\text{ms}}\|_{a_Q}^2\\
&+\frac{C_3}{4\epsilon\zeta} \ \sum_{1\leq 
i<N}b^{i}(w^i-u^i,w^i-u^i)+\frac{C_3}{4\epsilon\zeta} \| w-u_{\text{ms}} 
\|_{b}^2+\frac{C_3\zeta}{\epsilon} \|w-u\|_{a_Q}^2\\
&+\frac{C_3\zeta}{\epsilon} \| 
w-u_{\text{ms}} \|_{a_Q}^2+ \frac{C_3\epsilon}{2}\|w-u_{\text{ms}}\|_{a_{Q}}^2.
\end{split}
\end{equation*}
Rearrange the inequality and carefully choose $\epsilon$, $\zeta$, $\xi$, $\eta$ and let
$$K = {2\cdot (\frac{4(C_3)^4}{(\alpha^-)^3}+\frac{1}{2})}.$$
We obtain

\begin{equation}
\begin{split}
\frac{1}{2}\frac{d}{dt}\|w-u_{\text{ms}}\|_b^2 &-\frac{1}{2}K \|w-u_{\text{ms}} \|_{b}^2+ (\frac{\alpha^-}{4})\|w-u_{\text{ms}}\|^2_{a_Q}\\
&\leq \frac{1}{2}\|\frac{\partial (w-u)}{\partial t}\|_b^2 
+\frac{(\alpha^+)^2}{\alpha^-}\|w-u\|_{a_Q}^2+\frac{4C_3^2}{\alpha^-} \ 
\sum_{1\leq i<N}b^{i}(w^i-u^i,w^i-u^i).
\end{split}
\end{equation}
To get rid of term $\|u_{\text{ms}}-w\|^2_b$, we multiply a $e^{-Kt}\leq 1$ to the above inequality and integrate over $t$ from $0$ to $T$ for both sides, then we have
\begin{equation}\label{eq:mid}
	\begin{split}
	\frac{1}{2}\|w(T,\cdot)-u_{\text{ms}}(T,\cdot)\|_b^2  &+ \frac{\alpha^-\cdot e^{-KT}}{2}\int_0^T\|w-u_{\text{ms}}\|^2_{a_Q}\ dt\\
	&\leq \frac{1}{2}\int_0^T\|\frac{\partial (w-u)}{\partial t}\|_b^2\ dt +\frac{(\alpha^+)^2}{\alpha^-}\int_0^T\|w-u\|_{a_Q}^2\ dt\\
	&\quad+\frac{4(C_3)^2}{\alpha^-} \int_0^T\sum_{1\leq 
i<N}b^i(w^i-u^i,w^i-u^i)\ dt+
	\frac{1}{2}\|w(0,\cdot)-u_{\text{ms}}(0,\cdot)\|_b^2. 
	\end{split}
\end{equation} 
 
We further define initial value $u_{\text{ms}}(0, \cdot) \in V_{\text{ms}}$, s.t.
$$b(u_{\text{ms}}(0,\cdot), v) =b(u(0,\cdot), v)\quad \forall v\in V_{\text{ms}}. $$
Thus,
\begin{equation}\label{eq:tequal0}
\|w(0,\cdot)-u_{\text{ms}}(0,\cdot)\|_b\leq\|w(0,\cdot)-u(0,\cdot)\|_b.
\end{equation}

Making use of the Poincare Inequality, we also have for some constant $K_{3}>0$

%
\begin{equation}\label{eq:poincare}
\begin{split}
\ \sum_{1\leq i<N}b^{i}(w^i-u^i,w^i-u^i)
\leq K_3 \|w-u\|_{a_Q}^2.
\end{split}
\end{equation}

Combining \eqref{eq:mid}, \eqref{eq:tequal0} and \eqref{eq:poincare}, we conclude that 
there exist a constant $C_4>0$, such that 
\begin{equation}\label{eq:lemma1_mid_important}
\begin{split}
 \|w(T,\cdot)-u_{\text{ms}}(T,\cdot)\|_b^2  &+ \int_0^T\|w-u_{\text{ms}}\|^2_{a_Q}\ dt\\
&\leq \ C_4(\int_0^T\|\frac{\partial (w-u)}{\partial t}\|_b^2\ dt +\int_0^T\|w-u\|_{a_Q}^2\ dt+\|w(0,\cdot)-u(0,\cdot)\|_b^2 ).
\end{split}
\end{equation}

With \eqref{eq:lemma1_mid_important}, we can start derive the inequality for Theorem~\ref{thm1},
\begin{equation}\label{eq:lemma_mid}
\begin{split}
\|u(T,\cdot)-u_{\text{ms}}(T,\cdot)\|_b^2 & + \int_0^T \|u-u_{\text{ms}}\|_{a_Q}^2\ dt\\
 &\leq \|w(T,\cdot)-u(T,\cdot)\|_b^2 + \|w(T,\cdot)-u_{\text{ms}}(T,\cdot)\|_b^2\\
 &\quad+\int_0^T\|w-u\|_{a_Q}^2\ dt+\int_0^T\|w-u_{\text{ms}}\|_{a_Q}^2\ dt.
\end{split}	
\end{equation}
For the first term on the right hand side of Inequality \eqref{eq:lemma_mid}, we have
\begin{equation*}
\|w(T,\cdot)-u(T,\cdot)\|_b^2 
\leq\  2 \int_0^T\|\frac{\partial (w-u)}{\partial t}\|_b^2\ dt + 2\|w^i(0,\cdot)-u^i(0,\cdot)\|_b^2.
\end{equation*}
Combining the last estimate with \eqref{eq:lemma_mid} and \eqref{eq:lemma1_mid_important}, we conclude that 
for any $ w\in V_{\text{ms}}$, the inequality holds for a constant $C>0$, such that
\begin{equation}
\begin{split}
& \|u(T,\cdot)-u_{\text{ms}}(T,\cdot)\|_b^2 + \int_0^T \|u-u_{\text{ms}}\|_{a_Q}^2\ dt\\
\leq&\  C(\int_0^T\|\frac{\partial (w-u)}{\partial t}\|_b^2\ dt +\int_0^T\|w-u\|_{a_Q}^2\ dt+\|w(0,\cdot)-u(0,\cdot)\|_b^2 ).
\end{split}	
\end{equation}
This completes our proof.
 
\end{proof}

\subsection{Proof of Theorem~\ref{thm2}}
%

\begin{proof}
Since $u_{\text{snap}} \in V_{\text{snap}}$, we can write 
\begin{equation}\label{eq:u_snap_spectral}
u_{\text{snap}}(t, x) = \sum_j \sum_kc_k^{(j)}(t)\chi^{\omega_j}(x)\psi_{k}^{\omega_j}(x), 
\end{equation}

and we define the local component of $u_{\text{snap}}$ by 
\begin{equation}
\label{eq:u_snap_local}
u^{(j)}_{\text{snap}}(t, x) = \sum_k c_k^{(j)}(t)\psi_k^{\omega_j}(x).
\end{equation}
We define $w \in V_{\text{ms}}$ as the projection of $u_{\text{snap}}$ onto $V_{\text{ms}}$ by
\begin{equation}\label{eq:proj_w}
w=\sum_{j} \sum_{k =1}^{L_{\omega_j}}c_k^{(j)}(t)\psi_{k,ms}^{\omega_j}(x) = \sum_j \sum_{k=1}^{L_{\omega_j}}c_k^{(j)}(t)\chi^{\omega_j}(x)\psi_k^{\omega_j}(x).
\end{equation}
 From the definitions \eqref{eq:u_snap_spectral} and \eqref{eq:proj_w} ,we have 	
	\begin{equation}\label{eq:lemma_common}
		u_{\text{snap}} - w = \sum_j\sum_{k> L_{\omega_j}} c_k^{(j)}(t)\chi^{\omega_j}(x)\psi_k^{\omega_j}(x), 
	\end{equation}
The desired result follows from the estimates in 
Lemma~\ref{lemma2.1}, Lemma~\ref{lemma2.2} and Lemma~\ref{lemma2.3}.
\end{proof}

\begin{lemma}
\label{lemma2.1}
Let $u_{\text{snap}} \in V_{\text{snap}}$ be defined in \eqref{eq:u_snap} 
and $w \in V_{\text{ms}}$ be defined in \eqref{eq:proj_w}. 
Then there exists a constant $C>0$ such that 
\begin{equation}\label{lemma2:3}
\left\|\frac{\partial(u_{\text{snap}} - w )}{\partial t}\right\|_b^2 \leq 
\frac{C}{\Lambda}\left\|\frac{\partial u}{\partial t}\right\|_{a_Q}^2.
\end{equation}
\begin{proof}
	\begin{equation*}
		\frac{\partial(u_{\text{snap}} - w )}{\partial t}= \sum_j\sum_{k> L_{\omega_j}} (\frac{d}{dt}c_k^{(j)}(t))\chi^{\omega_j}(x)\psi_k^{\omega_j}(x) 
	\end{equation*}
Thus, for some constant $D_1>0$, we have 
	\begin{equation}\label{lemma2:1}
\|\frac{\partial(u_{\text{snap}} - w )}{\partial t}\|_b^2\\
\leq\   D_1 \sum_j\|\sum_{k> L_{\omega_j}} (\frac{d}{dt}c_k^{(j)}(t))\psi_k^{\omega_j}(x)\|_b^2, 
	\end{equation}
and the right-hand side can be estimated as 
\begin{equation*}
\begin{split}
&\|\sum_{k> L_{\omega_j}} (\frac{d}{dt}c_k^{(j)}(t))\psi_k^{\omega_j}(x)\|_b^2\\
 =\ & \sum_i\int_{\Omega_M} b^i (\sum_{k> L_{\omega_j}} (\frac{d}{dt}c_k^{(j)}(t))\psi_k^{i,\omega_j}(x))^2\ dx+
 \sum_i\sum_s\int_{\Omega_{F,s}} b_{F,s} (\sum_{k> L_{\omega_j}} (\frac{d}{dt}c_k^{(j)}(t))\psi_k^{i,\omega_j}(x))^2\ dx\\
\leq & D_2 [\sum_{1 \leq i < N}(\int_{\Omega_M}\frac{\kappa^i}{\mu} (\sum_{k> L_{\omega_j}} (\frac{d}{dt}c_k^{(j)}(t))\psi_k^{i,\omega_j}(x))^2\ dx+ \sum_s\int_{\Omega_{F,s}} \frac{\kappa_{F,s}}{\mu} (\sum_{k> L_{\omega_j}} (\frac{d}{dt}c_k^{(j)}(t))\psi_k^{i,\omega_j}(x))^2\ dx) \\
&+ (\int_{\Omega_M} (\sum_{k> L_{\omega_j}} (\frac{d}{dt}c_k^{(j)}(t))\psi_k^{N,\omega_j}(x))^2\ dx+ \sum_s\int_{\Omega_{F,s}}(\sum_{k> L_{\omega_j}} (\frac{d}{dt}c_k^{(j)}(t))\psi_k^{N,\omega_j}(x))^2\ dx)]\\
=\ & D_2 [\sum_{1\leq i<N}(\int_{\omega_{j, M}} \frac{\kappa^{i}}{\mu}  (\sum_{k> L_{\omega_j}} (\frac{d}{dt}c_k^{(j)}(t))\psi_k^{i,\omega_j}(x))^2\ dx+ \sum_s\int_{\omega_{j,F,s}} \frac{\kappa_{F,s}}{\mu}  (\sum_{k> L_{\omega_j}} (\frac{d}{dt}c_k^{(j)}(t))\psi_k^{i,\omega_j}(x))^2\ dx)\\
&+ (\int_{\omega_M} (\sum_{k> L_{\omega_j}} (\frac{d}{dt}c_k^{(j)}(t))\psi_k^{N,\omega_j}(x))^2\ dx+ \sum_s\int_{\omega_{F,s}}(\sum_{k> L_{\omega_j}} (\frac{d}{dt}c_k^{(j)}(t))\psi_k^{N,\omega_j}(x))^2\ dx)]\\
=\ &D_2 s^{(j)}(\sum_{k> L_{\omega_j}} (\frac{d}{dt}c_k^{(j)}(t))\psi_k^{\omega_j}(x), \sum_{k> L_{\omega_j}} (\frac{d}{dt}c_k^{(j)}(t))\psi_k^{\omega_j}(x))
\end{split}	
\end{equation*}
for some constant $D_2>0$.

By spectral problem \eqref{eq:spectral} and the orthogonality of eigenfunctions $\{\psi_k^{\omega_j}\}_k$, we have 
\begin{equation}
\begin{split}
& s^{(j)}(\sum_{k> L_{\omega_j}} (\frac{d}{dt}c_k^{(j)}(t))\psi_k^{\omega_j}(x), \sum_{k> L_{\omega_j}} (\frac{d}{dt}c_k^{(j)}(t))\psi_k^{\omega_j}(x))\\
\leq\ & \frac{1}{\lambda_{L_{\omega_j}+1}^{\omega_j}}a_Q ^{(j)}(\sum_{k> L_{\omega_j}} (\frac{d}{dt}c_k^{(j)}(t))\psi_k^{\omega_j}(x), \sum_{k> L_{\omega_j}} (\frac{d}{dt}c_k^{(j)}(t))\psi_k^{\omega_j}(x))\\
\leq\ & \frac{1}{\lambda_{L_{\omega_j}+1}^{\omega_j}}a_Q ^{(j)}(\sum_{k} (\frac{d}{dt}c_k^{(j)}(t))\psi_k^{\omega_j}(x), \sum_{k} (\frac{d}{dt}c_k^{(j)}(t))\psi_k^{\omega_j}(x))\\
 =\ & \frac{1}{\lambda_{L_{\omega_j}+1}^{\omega_j}}a_Q ^{(j)}(\frac{\partial u_{\text{snap}}^{(j)}}{\partial t}, \frac{\partial u_{\text{snap}}^{(j)}}{\partial t}).
\end{split}
\end{equation}
Substituting this equation back to \eqref{lemma2:1}, we obtain
\begin{equation}\label{lemma2:2}
	\begin{split}
	\|\frac{\partial(u_{\text{snap}} - w )}{\partial t}\|_b^2
	\leq D_1 D_2\sum_j \frac{1}{\lambda_{L_{\omega_j}+1}^{\omega_j}}a_Q ^{(j)}(\frac{\partial u_{\text{snap}}^{(j)}}{\partial t}, \frac{\partial u_{\text{snap}}^{(j)}}{\partial t}).\\
	\end{split}
\end{equation}
Since $u_{\text{snap}}$ is the projection of $u$ in each $\omega_j$ by definition \eqref{eq:u_snap_local}, so we have
$$a_Q^{j}(u^{(j)},v ) = a_Q^{j}(u_{\text{snap}}^{(j)},v )\qquad \forall v\in V_{\text{snap}}^{(j)} .$$
More specifically, let $v = u_{\text{snap}}^{(j)} $ we have 
\begin{equation*}
\begin{split}
a_Q^{j}(u_{\text{snap}}^{(j)},u_{\text{snap}}^{(j)} ) &=a_Q^{j}(u^{(j)},u_{\text{snap}}^{(j)} ),\\
\|u_{\text{snap}}^{(j)}\|_{a_Q}^2 &\leq \|u_{\text{snap}}^{(j)}\|_{a_Q} \|u^{(j)}\|_{a_Q}.
\end{split}
\end{equation*}
Therefore,
\begin{equation*}
a_Q^{j}(u_{\text{snap}}^{(j)},u_{\text{snap}}^{(j)} ) \leq a_Q^{j}(u^{(j)},u^{(j)} ).
\end{equation*}
Similarly,
\begin{equation}
a_Q^{j}(\frac{\partial u_{\text{snap}}^{(j)}}{\partial t},\frac{\partial u_{\text{snap}}^{(j)}}{\partial t}) \leq a_Q^{j}(\frac{\partial u^{(j)}}{\partial t},\frac{\partial u^{(j)}}{\partial t}).
\end{equation}
\begin{equation*}
\begin{split}
\end{split}
\end{equation*}
Thus, from \eqref{lemma2:2}, we have
\begin{equation}\label{lemma2:3}
\|\frac{\partial(u_{\text{snap}} - w )}{\partial t}\|_b^2 \leq D_1 D_2\sum_j \frac{1}{\lambda_{L_{\omega_j}+1}^{\omega_j}}a_Q^{j}(\frac{\partial u^{(j)}}{\partial t},\frac{\partial u^{(j)}}{\partial t})\\
 \leq \frac{D_1D_2}{\min_j\{\lambda_{L_{\omega_j}+1}^{\omega_j}\}}\|\frac{\partial u}{\partial t}\|_{a_Q}^2.
\end{equation}
This completes the proof.
\end{proof}
\end{lemma}

\begin{lemma}
\label{lemma2.20}
	For coupled multiscale basis function, if $u$ satisfies the following
	\begin{equation}
		\sum_{1\leq i<N}\int_{\omega_{j, M}} {\kappa^i}\nabla u^i\nabla v^i\ dx +	\sum_{1\leq i<N}\sum_s\frac{\kappa_{F,s}}{\mu}\int_{\omega_{j,F,s}}\kappa_{F,s} \nabla_F u^i\nabla_F v^i\ dx +\ q(u, v) = \int_{\omega} fv\ dx \qquad \forall v\in V_{\text{snap}}^{(j)},
	\end{equation}
there exists some constant $C$, such that	
	\begin{equation}\label{lemma_3_conclusion}
	\begin{split}
	&\sum_{{1\leq i<N}}\int_{\omega_{j, M}} {\kappa^i}(\chi^{\omega_j})^2(\nabla u^i)^2\ dx + 	\sum_{1\leq i<N}\sum_s\int_{\omega_{j,F,s}}{\kappa_{f,s}} (\chi^{\omega_j})^2(\nabla_F u^i)^2\ dx +q(\chi^{\omega_j}u, \chi^{\omega_j}u) \\
	\leq\ & C\{\sum_{1\leq i<N}[\int_{\omega_j} (f^i)^2\frac{(\chi^{\omega_j})^2}{|\nabla\chi^{\omega_j}|^2\kappa^i}\ dx+\int_{\omega_{j, M}}\kappa^{i} (u^i\nabla\chi^{\omega_j})^2\ dx +\sum_s\int_{\omega_{j,F,s}} {\kappa_{F,s}}(u^i\nabla_F\chi^{\omega_j})^2\ dx] \\
	&+ \int_{\omega_j} (f^N)^2\frac{(\chi^{\omega_j})^2}{|\nabla\chi^{\omega_j}|^2}\ dx \}.\\
	\end{split}
	\end{equation}
\end{lemma}
\begin{proof}
Let $v = (\chi^{\omega_j})^2u$ and obtain 
\begin{equation*}
\sum_{1\leq i<N}\int_{\omega_{j, M}}{\kappa^i} \nabla u^i\nabla ((\chi^{\omega_j})^2u^i)\ dx +	\sum_{1\leq i<N}\sum_s\int_{\omega_{j,F,s}}\frac{\kappa_{F,s}}{\mu} \nabla_F u^i\nabla_F((\chi^{\omega_j})^2u^i) \ dx +q(u, (\chi^{\omega_j})^2u) = \int_{\omega} f(\chi^{\omega_j})^2u\ dx.
\end{equation*}
This can be further rewrite as 
\begin{equation*}
	\begin{split}
&\sum_{1\leq i<N}\int_{\omega_{j, M}} \kappa^i(\chi^{\omega_j})^2(\nabla u^i)^2\ dx + 	\sum_{1\leq i<N}\sum_s\int_{\omega_{j,F,s}}\kappa_{F,s}  (\chi^{\omega_j})^2(\nabla_F u^i)^2\ dx +q(\chi^{\omega_j}u, \chi^{\omega_j}u) \\
=\ & \sum_{1\leq i<N}\int_{\omega} f^i\frac{(\chi^{\omega_j})^2}{\nabla\chi^{\omega_j}\sqrt{\kappa^i}}\sqrt{\kappa^i}u^i\nabla\chi^{\omega_j}\ dx+\int_{\omega} f^N\frac{(\chi^{\omega_j})^2}{\nabla\chi^{\omega_j}}u^N\nabla\chi^{\omega_j}\ dx\\
&\ -2\sum_{1\leq i<N}\int_{\omega_{j, M}}\kappa^i \nabla u^i\nabla \chi^{\omega_j} u^i \chi^{\omega_j}\ dx -2\sum_{1\leq i<N}\sum_s\int_{\omega_{j,F,s}}\kappa_{F,s} \nabla_F u^i\nabla_F \chi^{\omega_j}u^i\chi^{\omega_j}\ dx\\
\leq\ & \frac{\epsilon}{2}\sum_{1\leq i<N}\int_{\omega_j} (f^i)^2\frac{(\chi^{\omega_j})^4}{|\nabla\chi^{\omega_j}|^2\kappa^i}\ dx+\frac{1}{2\epsilon}\sum_{1\leq i<N}\int_{\omega_{j, M}}\kappa^i(u^i\nabla\chi^{\omega_j})^2\ dx\\
&+ \frac{\epsilon}{2}\int_{\omega_j} (f^N)^2\frac{(\chi^{\omega_j})^4}{|\nabla\chi^{\omega_j}|^2}\ dx+\frac{1}{2\epsilon}\int_{\omega_{j, M}}(u^N\nabla)^2\ dx\\
&+\sum_{1\leq i<N}\sum_s\frac{1}{2\epsilon}\int_{\omega_{j,F,s}}\kappa_{F,s}(u^i\nabla_F)^2\ dx + \sum_s\frac{1}{2\epsilon}\int_{\omega_{j,F,s}}(u^N\nabla_F)^2\ dx\\
&+\epsilon\sum_{1\leq i<N}\int_{\omega_{j, M}}\kappa^i (\chi^{\omega_j}\nabla u^i )^2\ dx +  \frac{1}{ \epsilon}\sum_{1\leq i<N}\int_{\omega_{j, M}}\kappa^i (u^i\nabla\chi^{\omega_j})^2\ dx\\
& +\epsilon\sum_{1\leq i<N}\sum_s\int_{\omega_{j,F,s}}\kappa_{F,s} (\chi^{\omega_j}\nabla_F u^i)^2\ dx +  \frac{1}{ \epsilon}\sum_{1\leq i<N}\sum_s\int_{\omega_{j,F,s}} \kappa_{F,s}(u^i\nabla_F\chi^{\omega_j})^2\ dx.
\end{split}
\end{equation*}
Let $\epsilon =1/2$ and rearrange the inequality. Then, for some constant $C>0$, we obtain the conclusion of \eqref{lemma_3_conclusion}.

\end{proof} 

\begin{lemma}
\label{lemma2.2}
Let $u_{\text{snap}} \in V_{\text{snap}}$ be defined in \eqref{eq:u_snap} 
and $w \in V_{\text{ms}}$ be defined in \eqref{eq:proj_w}. 
Then there exists a constant $C>0$ such that 
	\begin{equation}
		\int_0^T \|w-u_{\text{snap}}\|_{a_Q}^2\ dt\leq \frac{C}{\Lambda}\int_0^T \|u\|_{a_Q}^2.
	\end{equation}
\end{lemma}
\begin{proof}
By \eqref{eq:lemma_common}, we have 
\begin{equation}\label{lemma3:1}
	\|w-u_{\text{snap}}\|_{a_Q}^2 = \|\sum_j\sum_{k> L_{\omega_j}} c_k^{(j)}(t)\chi^{\omega_j}(x)\psi_k^{\omega_j}(x) \|_{a_Q}^2
	\leq N_v\sum_j\|\chi^{\omega_j}(x)\sum_{k>L_{\omega_j}}c_k^{(j)}(t)\psi_k^{\omega_j}(x)\|_{a_Q}^2.
\end{equation}
Let 
$$e^{(j)} = \sum_{k>L_{\omega_j}}c_k^{(j)}(t)\psi_k^{\omega_j}(x),$$
then
\begin{equation*}
\begin{split}
&\|\chi^{\omega_j}(x)e^{(j)}\|_{a_Q}^2\\
=\ & \sum_{1\leq i<N}\int_{\omega_{j, M}}\frac{\kappa^i}{\mu} (\chi^{\omega_j})^2[\nabla e^{(j),i}]^2\ dx+ \sum_{1\leq i<N}\int_{\omega_{j, M}}\frac{\kappa^i}{\mu} (\nabla\chi^{\omega_j})^2[e^{(j),i}]^2\ dx\\
 &+ \sum_{1\leq i<N}\sum_s\int_{\omega_{j,F,s}}\frac{\kappa_{F,s}}{\mu} [\nabla_F(\chi^{\omega_j})]^2[e^{(j),i}]^2\ dx+ \sum_{1\leq i<N}\sum_s\int_{\omega_{j,F,s}}\frac{\kappa_{F,s}}{\mu} (\chi^{\omega_j})^2[e^{(j),i}]^2\ dx\\
 &+q(\chi^{\omega_j}(x)e^{(j)}, \chi^{\omega_j}(x)e^{(j)}),
\end{split}
\end{equation*}
where 
\begin{equation*}
\begin{split}
&\sum_{1\leq i<N}\int_{\omega_{j, M}}\frac{\kappa^i}{\mu} (\nabla\chi^{\omega_j})^2[e^{(j),i}]^2\ dx + \sum_{1\leq i<N}\sum_s\int_{\omega_{j,F,s}}\frac{\kappa_{F,s}}{\mu} [\nabla_F(\chi^{\omega_j})]^2[e^{(j),i}]^2\ dx \\
\leq\ &D_3\sum_{1\leq i<N}\int_{\omega_{j, M}}\frac{\kappa^i}{\mu} [e^{(j),i}]^2\ dx + D_3 \sum_{1\leq i<N}\sum_s\int_{\omega_{j,F,s}}\frac{\kappa_{F,s}}{\mu} [e^{(j),i}]^2\ dx\\
\leq\ &D_3 s^{(j)} (e^{(j)}, e^{(j)})\\
\end{split}
\end{equation*}
for some constant $D_3$.
From Lemma~\ref{lemma2.20}, there exists some constant $D_4$ such that 
\begin{equation*}
\begin{split}
&\sum_{1\leq i<N}\int_{\omega_{j, M}}\frac{\kappa^i}{\mu} (\chi^{\omega_j})^2[\nabla e^{(j),i}]^2\ dx+\sum_{1\leq i<N}\sum_s\int_{\omega_{j,F,s}}\frac{\kappa_{F,s}}{\mu} (\chi^{\omega_j})^2[\nabla_F e^{(j),i}]^2\ dx + q(\chi^{\omega_j}(x) e^{(j)}, \chi^{\omega_j}(x)e^{(j)})\\ 
\leq\ &D_4 [\sum_{1\leq i<N}\int_{\omega_{j,M}}\frac{\kappa^i}{\mu}|\nabla\chi^{\omega_j}|^2(e^{(j)})^2\ dx + \sum_{1\leq i<N}\sum_s\int_{\omega_{j,F,s}}\frac{\kappa_{F,s}}{\mu}|\nabla_F\chi^{\omega_j}|^2(e^{(j)})^2\ dx]\\
\leq\ &D_3D_4\ s^{(j)}(e^{(j)}, e^{(j)}).
\end{split}
\end{equation*}
By bilinearity of $ a^{(j)}$ and $s^{(j)}$ as well as the orthogonality of $\{\psi_k^{\omega_j}\}_k$ ,we finally have
\begin{equation}
\begin{split}
	&\|w-u_{\text{snap}}\|_{a_Q}^2  \leq N_v\sum_j\|\chi^{\omega_j}(x)e^{(j)}\|_{a_Q}^2\leq D_5\sum_j s^{(j)}(e^{(j)}, e^{(j)})\\
	\leq\ & D_5\sum_j \frac{1}{\lambda_{L_{\omega_j}+1}^{\omega_j}}a_Q^{(j)}(e^{(j)}, e^{(j)})\leq \frac{D_5}{\Lambda} a_Q(u_{\text{snap}}, u_{\text{snap}}) = \frac{D_5}{\Lambda} \|u_{\text{snap}}\|_{a_Q}^2,
\end{split}
\end{equation}
 for a properly selected constant $D_5$.
	
\end{proof}
\begin{lemma}
\label{lemma2.3}
Let $u_{\text{snap}} \in V_{\text{snap}}$ be defined in \eqref{eq:u_snap} 
and $w \in V_{\text{ms}}$ be defined in \eqref{eq:proj_w}. 
Then there exists a constant $C>0$ such that 
	\begin{equation}
	\|w(0,\cdot)-u_{\text{snap}}(0,\cdot)\|_b^2\leq\frac{C}{\Lambda}\|u(0,\cdot)\|_{a_Q}^2.
	\end{equation}
\end{lemma}
\begin{proof}
	Using a similar idea as in Lemma~\ref{lemma2.2}, we let
	$$e_0^{(j)} = \sum_{k>L_{\omega_j}}c_k^{(j)}(0)\psi_k^{\omega_j}(x).$$
	Then we have
	\begin{equation}
	\begin{split}
&\|u_{\text{snap}}(0,\cdot)-w(0,\cdot)\|_b^2 =\|\sum_j\chi^{\omega_j}(x)\sum_{k>L_{\omega_j}}c_k^{(j)}(0)\psi_k^{\omega_j}(x) \|_b^2= \|\sum_j \chi^{\omega_j}(x)e_0^{(j)}\|_b^2 \\
\leq &\ D_1 \sum_j\|e_0^{(j)}\|_b^2\leq D_1D_2 \sum_j s^{(j)}(e_0^{(j)}, e_0^{(j)})\leq  D_1D_2 \frac{1}{\Lambda}\sum_j a_Q^{(j)}(e_0^{(j)}, e_0^{(j)})\\
\leq &  D_1D_2 \frac{1}{\Lambda} \sum_ja_Q^{(j)}(u_{\text{snap}}^{(j)}(0,x), u^{(j)}_{\text{snap}}(0,x)) =    D_1D_2 \frac{1}{\Lambda} \sum_ja_Q^{(j)}(u^{(j)}(0,x), u^{(j)}(0,x))\\
=\ &D_1D_2 \frac{1}{\Lambda}\| u^{(j)}(0,\cdot)\|_{a_Q}^2.
	\end{split}
	\end{equation}
	This completes the proof.
	
\end{proof}

\section{List of parameters}

\begin{table}
	\centering
	\begin{tabular}{|c|c|c|}
		\hline
		Quantity	&Description &Unit\\
		\hline
		$\phi$	&Porosity &fraction\\
		$c$& Compressibility of flow	&kP$a^{-1}$\\
		$\mu$&Viscosity of liquid	&P$a\cdot$s\\
		$\kappa$ & Permeability & $\mu$m$^2$\\
		$u$ &Pressure & kPa\\
		$u^0$ &Reference Pressure& kPa\\
		$B$ &FVF& m$^3/$m$^3$\\
		$B^{\circ}$& FVF at $u^{0}$	&m$^3/$m$^3$\\
		$q_{sc}$ &Source & m$^3/$day\\
		$\delta$ &Shape factor & m$^{-2}$\\
		\hline
	\end{tabular}
	
	\caption{Units used for all quantities}
	\label{tab:units}
\end{table}

\begin{table}
	\centering
	\begin{tabular}{|c|c|}
		\hline
		Quantity	&Value\\
		\hline
		Size of model& $15000ft\times 15000ft$\\
		$\phi^m$	&0.2\\
		$\phi^f$	&0.01\\
		$\phi^v$	&0.1\\
		$\phi_{F,j}$	&1\\
		$\kappa^f$ & $10^{-12}$\\
		$\kappa^v$ & $10^{-13}$\\
		$\kappa_{F,j}$ & $8.2606\times10^{-8}$\\
		$c$	&$1.4504\times10^{-8}$\\
		$\mu$	&$8\times 10^{-3}$\\
		$u^0$ &$ 2.0684\times 10^7$\\
		$B^{\circ}$	&1.1\\
		$\delta$ &
		1/$h_{\text{min}}^2$\\
		\hline
	\end{tabular}
	
	\caption{Values of all quantities}
	\label{tab:values}
\end{table}

\bibliographystyle{plain}
\bibliography{references}

\begin{thebibliography}{10}

\bibitem{akkutlu2017multiscale}
I~Yucel Akkutlu, Yalchin Efendiev, Maria Vasilyeva, and Yuhe Wang.
\newblock Multiscale model reduction for shale gas transport in a coupled
  discrete fracture and dual-continuum porous media.
\newblock {\em Journal of Natural Gas Science and Engineering}, 48:65--76,
  2017.

\bibitem{akkutlu2018multiscale}
I~Yucel Akkutlu, Yalchin Efendiev, Maria Vasilyeva, and Yuhe Wang.
\newblock Multiscale model reduction for shale gas transport in poroelastic
  fractured media.
\newblock {\em Journal of Computational Physics}, 353:356--376, 2018.

\bibitem{arbogast1990derivation}
Todd Arbogast, Jim Douglas, Jr, and Ulrich Hornung.
\newblock Derivation of the double porosity model of single phase flow via
  homogenization theory.
\newblock {\em SIAM Journal on Mathematical Analysis}, 21(4):823--836, 1990.

\bibitem{baca1984modelling}
RG~Baca, RC~Arnett, and DW~Langford.
\newblock Modelling fluid flow in fractured-porous rock masses by
  finite-element techniques.
\newblock {\em International Journal for Numerical Methods in Fluids},
  4(4):337--348, 1984.

\bibitem{barenblatt1960basic}
GI~Barenblatt, Iu~P Zheltov, and IN~Kochina.
\newblock Basic concepts in the theory of seepage of homogeneous liquids in
  fissured rocks [strata].
\newblock {\em Journal of applied mathematics and mechanics}, 24(5):1286--1303,
  1960.

\bibitem{chai2018efficient}
Z~Chai, B~Yan, JE~Killough, and Y~Wang.
\newblock An efficient method for fractured shale reservoir history matching:
  The embedded discrete fracture multi-continuum approach.
\newblock {\em Journal of Petroleum Science and Engineering}, 160:170--181,
  2018.

\bibitem{chung2017coupling}
Eric~T Chung, Yalchin Efendiev, Tat Leung, and Maria Vasilyeva.
\newblock Coupling of multiscale and multi-continuum approaches.
\newblock {\em GEM-International Journal on Geomathematics}, 8(1):9--41, 2017.

\bibitem{durlofsky1991numerical}
Louis~J Durlofsky.
\newblock Numerical calculation of equivalent grid block permeability tensors
  for heterogeneous porous media.
\newblock {\em Water resources research}, 27(5):699--708, 1991.

\bibitem{efendiev2013generalized}
Yalchin Efendiev, Juan Galvis, and Thomas~Y Hou.
\newblock Generalized multiscale finite element methods (gmsfem).
\newblock {\em Journal of Computational Physics}, 251:116--135, 2013.

\bibitem{efendiev2011multiscale}
Yalchin Efendiev, Juan Galvis, and Xiao-Hui Wu.
\newblock Multiscale finite element methods for high-contrast problems using
  local spectral basis functions.
\newblock {\em Journal of Computational Physics}, 230(4):937--955, 2011.

\bibitem{efendiev2015hierarchical}
Yalchin Efendiev, Seong Lee, Guanglian Li, Jun Yao, and Na~Zhang.
\newblock Hierarchical multiscale modeling for flows in fractured media using
  generalized multiscale finite element method.
\newblock {\em GEM-International Journal on Geomathematics}, 6(2):141--162,
  2015.

\bibitem{farago1991finite}
Istv{\'a}n Farago.
\newblock Finite element method for solving nonlinear parabolic equations.
\newblock {\em Computers \& Mathematics with Applications}, 21(1):59--69, 1991.

\bibitem{hassani1998review}
Behrooz Hassani and Ernest Hinton.
\newblock A review of homogenization and topology optimization
  i—homogenization theory for media with periodic structure.
\newblock {\em Computers \& Structures}, 69(6):707--717, 1998.

\bibitem{karimi2001numerical}
Mohammad Karimi-Fard, Abbas Firoozabadi, et~al.
\newblock Numerical simulation of water injection in 2d fractured media using
  discrete-fracture model.
\newblock In {\em SPE annual technical conference and exhibition}. Society of
  Petroleum Engineers, 2001.

\bibitem{lee1999efficient}
SH~Lee, CL~Jensen, MF~Lough, et~al.
\newblock An efficient finite difference model for flow in a reservoir with
  multiple length-scale fractures.
\newblock In {\em SPE Annual Technical Conference and Exhibition}. Society of
  Petroleum Engineers, 1999.

\bibitem{li2006efficient}
Liyong Li, Seong~Hee Lee, et~al.
\newblock Efficient field-scale simulation for black oil in a naturally
  fractured reservoir via discrete fracture networks and homogenized media.
\newblock In {\em International oil \& gas conference and exhibition in China}.
  Society of Petroleum Engineers, 2006.

\bibitem{li2018multiscale}
Qiuqi Li, Yuhe Wang, and Maria Vasilyeva.
\newblock Multiscale model reduction for fluid infiltration simulation through
  dual-continuum porous media with localized uncertainties.
\newblock {\em Journal of Computational and Applied Mathematics}, 336:127--146,
  2018.

\bibitem{matache2002two}
Ana-Maria Matache and Christoph Schwab.
\newblock Two-scale fem for homogenization problems.
\newblock {\em ESAIM: Mathematical Modelling and Numerical Analysis},
  36(4):537--572, 2002.

\bibitem{noorishad1982upstream}
Jahan Noorishad and Mohsen Mehran.
\newblock An upstream finite element method for solution of transient transport
  equation in fractured porous media.
\newblock {\em Water Resources Research}, 18(3):588--596, 1982.

\bibitem{pruess1982practical}
Karsten Pruess and TN~Narasimhan.
\newblock A practical method for modeling fluid and heat flow in fractured
  porous media.
\newblock 1982.

\bibitem{thomas1983fractured}
L~Kent Thomas, Thomas~N Dixon, Ray~G Pierson, et~al.
\newblock Fractured reservoir simulation.
\newblock {\em Society of Petroleum Engineers Journal}, 23(01):42--54, 1983.

\bibitem{wang2017generalized}
Min Wang, Chenji Wei, Hongqing Song, Yalchin Efendiev, Yuhe Wang, et~al.
\newblock Generalized multiscale coupling of triple-continuum model and
  discrete fracture network for carbonate reservoir simulation.
\newblock In {\em SPE Annual Technical Conference and Exhibition}. Society of
  Petroleum Engineers, 2017.

\bibitem{warren1963behavior}
JE~Warren, P~Jj Root, et~al.
\newblock The behavior of naturally fractured reservoirs.
\newblock 1963.

\bibitem{yan2016general}
Bicheng Yan, Masoud Alfi, Cheng An, Yang Cao, Yuhe Wang, and John~E Killough.
\newblock General multi-porosity simulation for fractured reservoir modeling.
\newblock {\em Journal of Natural Gas Science and Engineering}, 33:777--791,
  2016.

\bibitem{yan2016beyond}
Bicheng Yan, Yuhe Wang, and John~E Killough.
\newblock Beyond dual-porosity modeling for the simulation of complex flow
  mechanisms in shale reservoirs.
\newblock {\em Computational Geosciences}, 20(1):69--91, 2016.

\bibitem{yao2010discrete}
Jun Yao, Zhaoqin Huang, Yajun Li, Chenchen Wang, Xinrui Lv, et~al.
\newblock Discrete fracture-vug network model for modeling fluid flow in
  fractured vuggy porous media.
\newblock In {\em International oil and gas conference and exhibition in
  China}. Society of Petroleum Engineers, 2010.

\bibitem{zhang2018multiscale}
Na~Zhang, Yating Wang, Qian Sun, and Yuhe Wang.
\newblock Multiscale mass transfer coupling of triple-continuum and discrete
  fractures for flow simulation in fractured vuggy porous media.
\newblock {\em International Journal of Heat and Mass Transfer}, 116:484--495,
  2018.

\end{thebibliography}

\end{document}